\RequirePackage{fix-cm}

\let\oldvec\vec

\documentclass[smallextended, numbook, envcountsame]{svjour3}		

\let\vec\oldvec

\smartqed  
\usepackage{graphicx}

\usepackage{amsmath,amssymb,amsfonts}%
\usepackage{color}%
\usepackage{xcolor}  
\usepackage{a4}%
\usepackage{enumitem}%
\usepackage{verbatim}%
\usepackage{theorem}%
\usepackage[bookmarks,bookmarksopen,bookmarksdepth=3]{hyperref}

\usepackage[toc]{appendix}  

\newtheorem{nnremark}[theorem]{\bf Remark}
\renewenvironment{remark}{\begin{nnremark} \rm }{\hfill \hspace*{1pt}\hfill $\lhd$\end{nnremark}}

\spdefaulttheorem{assumption}{Assumption}{\bf}{\it}

\newcommand{\tm}{\times}%
\newcommand{\Uc}{\mathcal{U}}%
\newcommand{\R}{\mathbb{R}}%
\newcommand{\ccat}[3]{{#1\, \underset{#3}{\lozenge}\,{#2}}}%
\newcommand{\PD}{\mathcal{P}}%
\newcommand{\K}{\mathcal{K}}%
\newcommand{\Kinf}{\mathcal{K_\infty}}%
\newcommand{\KL}{\mathcal{KL}}%
\newcommand{\LL}{\mathcal{L}}%
\newcommand{\ep}{\varepsilon}%
\newcommand{\N}{\mathbb{N}}%
\newcommand{\PC}{\mathrm{PC}}%
\newcommand{\UGS}{\mathrm{UGS}}%
\newcommand{\Z}{\mathbb{Z}}%
\newcommand{\id}{\mathrm{id}}%
\newcommand{\dist}{\mathrm{dist}}%

\newcommand \diag  {\operatorname{diag}}

\newcounter{syscounter}

\newcommand \Iff   {\Leftrightarrow}

\usepackage{csquotes}  
\newcommand\q{\enquote}

\newcommand{\inner}{\mathrm{int}}%
\newcommand{\rmd}{\mathrm{d}}%
\journalname{Mathematics of Control, Signals, and Systems}

\begin{document}

\title{Nonlinear small-gain theorems for input-to-state stability of infinite interconnections}%

\titlerunning{Nonlinear small-gain theorems for ISS of infinite interconnections}%

\author{Andrii Mironchenko  \and Christoph Kawan \and Jochen Gl\"uck }



\institute{Andrii Mironchenko (corresponding author)  \at
							Faculty of Computer Science and Mathematics, University of Passau, Innstra\ss e 33, 94032 Passau, Germany\\
              \email{andrii.mironchenko@uni-passau.de}            
           \and
           Christoph Kawan \at Institute of Informatics, LMU Munich, {Oettingenstra{\ss}e 67, 80538 M\"{u}nchen}, Germany\\ \email{christoph.kawan@lmu.de} 
						\and
							Jochen Glück  \at
							Faculty of Computer Science and Mathematics, University of Passau, Innstra\ss e 33, 94032 Passau, Germany\\
              \email{jochen.glueck@uni-passau.de}            
}

\date{Received: date / Accepted: date}

\maketitle

\begin{abstract}
We consider infinite heterogeneous networks, consisting of input-to-state stable subsystems of possibly infinite dimension. We show that the network is input-to-state stable, provided that the gain operator satisfies a certain small-gain condition. We show that for finite networks of nonlinear systems this condition is equivalent to the so-called strong small-gain condition of the gain operator
 (and thus our results extend available results for finite networks), and for infinite networks with a linear gain operator they correspond to the condition that the spectral radius of the gain operator is less than one. We provide efficient criteria for input-to-state stability of infinite networks with linear gains, governed by linear and homogeneous gain operators, respectively.

\end{abstract}

\section{Introduction}

We {live} in a 
hyperconnected world, where the size of networks and the number of connections between their components are rapidly growing. Emerging technologies such as the Internet of Things, Cloud Computing, 5G communication, and so on make this trend even more distinct. Such complex networked systems include smart grids, connected vehicles, swarm robotics, and smart cities in which the participating agents may be plugged in and out from the network at any time.%

The unknown and possibly time-varying size of such networks poses new challenges for stability analysis and control design. One of the promising approaches to this problem is to over-approximate the network by an infinite network, and perform the stability analysis and control {design} for this infinite over-approximation \cite{CIZ09,JoB05b,DaD03}. This approach has received {significant attention} during the last two decades. In particular, a large body of literature is devoted to spatially invariant systems and/or linear systems consisting of an infinite number of components, interconnected with each other by means of the same pattern, see, e.g., \cite{BPD02,BaV05,BeJ17,CIZ09}.%

{
\textbf{State of the art: ISS theory.} ISS theory {was} initiated in \cite{Son89}, and has quickly become one of the pillars of nonlinear control theory, including robust stabilization, nonlinear observer design and analysis of large-scale networks, see \cite{KKK95,ArK01,Son08}. A tremendous progress in the development of the infinite-dimensional ISS theory has brought a number of powerful techniques for the robust stability analysis of distributed parameter systems, including Lyapunov methods \cite{DaM13,TPT17,ZhZ18}, characterizations of ISS \cite{MiW18b} for nonlinear systems, functional-analytic tools \cite{JNP18,JSZ17}, and spectral-based methods \cite{KaK16b,KaK19} for linear systems, see~\cite{MiP20} for a recent survey on this topic.  
}
The \emph{ISS small-gain approach} is especially fruitful for the analysis of coupled systems. In this method, the influence of any subsystem on other subsystems of a network is characterized by so-called gain functions. The gain operator constructed from these functions characterizes the interconnection structure of the network. The small-gain theorems for {interconnections of finitely many input-to-state stable systems governed by ordinary differential equations (ODEs)} \cite{JTP94,JMW96,DRW07,DRW10} state that if the gains are small enough, i.e., the gain operator satisfies a small-gain condition, the network is stable.%

Within the infinite-dimensional ISS theory, generalizations of these results to couplings of {finitely many} infinite-dimensional systems have been proposed in \cite{BLJ18,DaM13,Mir19b,TWJ12}. We refer to \cite{Mir19b} for more details and references on small-gain results for finite couplings. \emph{For the case of trajectory-based ISS small-gain theorems for finite networks, the main difficulties in going from finite to infinite dimensions stem from the fact that the characterizations of ISS developed for ODE systems in \cite{SoW96} are no more valid for infinite-dimensional systems. {As argued in \cite{Mir19b}, more general characterizations shown in \cite{MiW18b} have to be used, which requires major changes in the proof of the small-gain result.}}

Small-gain theorems for finite networks have been applied to the stability analysis of coupled parabolic-hyperbolic partial differential equations (PDEs) in \cite{KaK18}. Small-gain based boundary feedback design for global exponential stabilization of 1-D semilinear parabolic PDEs has been proposed in \cite{KaK19b}.%

\textbf{State of the art: infinite networks.}	On the other hand, recently a number of works appeared, devoted to stability and control of \emph{nonlinear infinite networks of ordinary differential equations, which are not necessarily spatially invariant}, see, e.g., \cite{DMS19a,DaP19,KMS19}. Small-gain analysis of infinite networks is especially challenging since the gain operator, collecting the information about the internal gains, acts on an infinite-dimensional space, in contrast to couplings of {finitely many} systems of arbitrary nature. This calls for a careful choice of the infinite-dimensional state space of the overall network, and motivates the use of the theory of positive operators on ordered Banach spaces for the small-gain analysis.%

In~\cite{DaP19}, it is shown that a countably infinite network of continuous-time input-to-state stable systems is ISS, provided that the gain functions capturing the influence of the subsystems on each other are all less than the identity, which is a very conservative condition. In~\cite{DMS19a}, it was shown that classic max-form strong small-gain conditions developed for finite networks in \cite{DRW10} do not ensure stability of infinite networks, even {in the linear case}. To address this issue, more restrictive robust strong small-gain conditions {have been} developed in~\cite{DMS19a}, but still the main results in \cite{DMS19a} have been shown under the quite strong restriction that there is a linear path of strict decay for the gain operator, which makes the result not fully nonlinear.%


In contrast, for networks consisting of exponentially ISS systems, possessing exponential ISS Lyapunov functions with linear gains, it was shown in \cite{KMS19} that if the spectral radius of the gain operator is less than one, then the whole network is exponentially ISS and there is a coercive exponential ISS Lyapunov function for the whole network. This result is tight and provides a complete generalization of \cite[Prop.~3.3]{DIW11} from finite to infinite networks.%



		%
	%

\textbf{Contribution.} 
\emph{In this work, we provide nonlinear ISS small-gain theorems for continuous-time infinite networks, whose components may be infinite-dimensional systems of different types. We do not impose any linearity and/or contractivity assumption for the gains, which makes the results truly nonlinear. Moreover, we do not restrict ourselves to couplings of ODE systems, but instead develop a framework, which allows for couplings of heterogeneous infinite-dimensional systems, which is important in the context of ODE-PDE, delay-PDE and PDE-PDE cascades.} We derive our small-gain theorems for uniform global stability (UGS) and ISS properties  
in the trajectory formulation, in contrast to the papers \cite{DMS19a,DaP19,KMS19}, where {the Lyapunov formulation was used}.%

We start by introducing a general class of infinite-dimensional control systems, which includes many classes of evolution PDEs, time-delay systems, ODEs, infinite switched systems, etc. Next, we introduce the concept of infinite interconnections for systems of this class, extending the framework developed in \cite{KaJ07,Mir19b}.%

{
\emph{Theorems \ref{thm_ugs_semimax_sg} and \ref{thm_ugs_summation_sg} are our small-gain results for uniform global stability of infinite networks. They use the monotone bounded invertibility (MBI) property of the gain operator}}, which is equivalent for finite networks (see Proposition~\ref{prop:small-gain-condition-n-dim-general}) to the strong small-gain condition, employed in the small-gain analysis of finite networks in \cite[Thm.~8]{DRW07} and \cite{Mir19b}. The proof of this result is based on the proof of the corresponding result for finite networks, see \cite[Thm.~8]{DRW07}.%

{
\emph{Theorems \ref{thm_smallgain_iss_semimax} and \ref{thm_smallgain_iss_summation} are our ISS small-gain results for infinite networks}} in semi-maximum and summation formulation, which state that an infinite network consisting of ISS systems is ISS provided that the discrete-time system induced by the gain operator has the so-called \emph{monotone limit (MLIM) property}. This property concerns the input-to-state behavior of the discrete-time control system $x(k+1) \leq \Gamma(x(k)) + u(k)$ induced by the gain operator $\Gamma$ and is implied by ISS of this system for monotonically decreasing solutions and in turn implies the monotone bounded invertibility property.%

In Section~\ref{sec:Small-gain conditions}, we analyze the MBI and MLIM properties, which {are employed in} the small-gain analysis of infinite networks of ISS systems. \emph{In Section~\ref{sec:A uniform small-gain condition and the MBI property}, we characterize the MBI property in terms of the uniform small-gain condition}, which is a uniform version of the classical small-gain condition $\Gamma(x) \not\geq x$ for all $x \geq 0$.%

In Section~\ref{sec:Non-uniform small-gain conditions}, we relate the uniform small-gain condition to the \emph{strong and robust strong small-gain conditions}, which have already been exploited in the small-gain analysis of finite \cite{DRW07,DRW10} and infinite \cite{DMS19a} networks. \emph{In Section~\ref{sec:Finite-dimensional systems}, we show in Proposition~\ref{prop:small-gain-condition-n-dim-general} that the uniform and strong small-gain conditions as well as the MBI and MLIM properties are equivalent for finite-dimensional nonlinear systems, if the gain operator {is of} summation or {max}-type.} {As a consequence of Proposition~\ref{prop:small-gain-condition-n-dim-general}, we see that our results extend those of \cite{Mir19b}, and thus also the classical small-gain theorems for finitely many finite-dimensional systems from \cite{DRW07}, even with minimal regularity assumptions on the interconnection (we require well-posedness {and the} BIC property only).}

{
In Appendix~\ref{sec:Exponential ISS of discrete-time systems}, we derive a characterization of {exponential ISS (eISS)} for discrete-time systems with a generating and normal cone, induced by homogeneous of degree one and subadditive operators (Proposition~\ref{prop:eISS-criterion-homogeneous-systems}). We apply this and recent results in \cite{GlM20}, to show in \emph{Proposition~\ref{prop:eISS-criterion-linear-systems} that for linear infinite-dimensional systems with a generating and normal cone the MBI, MLIM and the uniform small-gain condition all are equivalent to the spectral small-gain condition (saying that the spectral radius of the gain operator is less than one)}.%


Finally, in Appendix~\ref{sec:Systems, governed by a max-form gain operator}, we study relations between various uniform and non-uniform small-gain conditions for max-form gain operators with nonlinear gains, which are of particular importance in small-gain theory. Following \cite{DMS19a}, we study also the properties of the strong transitive closure of the gain operator. We use these properties to show (in Proposition~\ref{prop:join-morphism-eISS-criterion}) the equivalence of the MBI property, the MLIM property, and the existence of a path of strict decay for the max-form gain operator with linear gains. The results of that section are important for the development of linear and nonlinear Lyapunov-based small-gain theorems for infinite networks.%

Propositions~\ref{prop:eISS-criterion-linear-systems}, \ref{prop:join-morphism-eISS-criterion} and \ref{prop:eISS-criterion-homogeneous-systems} are useful, in particular, to obtain efficient small-gain theorems for infinite networks with linear gains, see Corollaries~\ref{cor:SGT-sup-linear-gains},~\ref{cor:SGT-sum-linear-gains}.
}

\section{Preliminaries}

\paragraph{Notation.} 

We write $\R$ for the real numbers, $\Z$ for the integers, and $\N = \{1,2,3,\ldots\}$ for the natural numbers. $\R_+$ and $\Z_+$ denote the sets of nonnegative reals and integers, respectively.%

We use the following classes of comparison functions:%
\begin{align*}
  \K &:= \left\{\gamma:\R_+ \to \R_+: \ \gamma\mbox{ is continuous and strictly increasing, }\gamma(0)=0\right\} \\
  \K_{\infty} &:= \left\{\gamma\in\K:\ \gamma\mbox{ is unbounded}\right\} \\
  \LL &:= \big\{\gamma:\R_+ \to \R_+:\ \gamma\mbox{ is continuous and strictly decreasing with}
 \lim\limits_{t\rightarrow\infty}\gamma(t)=0 \big\} \\
  \KL &:= \{\beta:\R_+^2 \rightarrow \R_+: \ \beta\mbox{ is continuous and } \\
	         &\qquad\qquad \qquad \qquad \qquad \beta(\cdot,t)\in{\K},\ \forall t \geq 0,\  \beta(r,\cdot)\in {\LL},\ \forall r > 0\}.
\end{align*}
For a normed linear space $(W,\|\cdot\|_W)$ and any $r>0$, we write $B_{r,W} :=\{w \in W: \|w\|_W < r\}$ (the open ball of radius $r$ around $0$ in $W$). By $\overline{B}_{r,W}$ we denote the corresponding closed ball.  If the space $W$ is clear from the context, we simply write $B_r$ and $\overline{B}_r$, respectively. For any nonempty set $S \subset W$ and any $x \in W$, we denote the distance from $x$ to $S$ by $\dist(x,S):=\inf_{y \in S}\|x-y\|_W$.%

For a set $U$, we let $U^{\R_+}$ denote the space of all maps from $\R_+$ to $U$. {For a nonempty set $J\subset \R_+$, we denote by $\|w\|_{J}$ the sup-norm of a bounded function $w:J \rightarrow W$, i.e., $\|w\|_{J} = \sup_{s\in J}\|w(s)\|_W$.} Given a nonempty index set $I$, we write $\ell_{\infty}(I)$ for the Banach space of all functions $x:I \rightarrow \R$ with $\|x\|_{\ell_{\infty}(I)} := \sup_{i\in I}|x(i)| < \infty$. Moreover, $\ell_{\infty}(I)^+ := \{ x \in \ell_{\infty}(I) : x(i) \geq 0,\ \forall i \in I \}$. We write ${\bf 1}$ for the vector in $\ell_{\infty}(I)^+$ whose components are all equal to $1$. If $I = \N$, we simply write $\ell_{\infty}$ and $\ell_{\infty}^+$, respectively. By $e_i$, $i\in I$, we denote the $i$-th unit vector in $\ell_{\infty}(I)$.%

Throughout the paper, all considered vector spaces are vector spaces over $\R$.%

\paragraph{Ordered vector spaces and positive operators.}

In the following, $X$ always denotes a real vector space. For two sets $A,B \subset X$, we write $A + B = \{a + b : a \in A, b \in B\}$, $-A = \{-a : a \in A\}$, and $\R_+ \cdot A = \{ r \cdot a : a \in A, r \in \R_+ \}$.%

Recall that a \emph{partial order} on a set $X$ is a relation on $X$ which is reflexive, transitive and antisymmetric. A subset $X^+ \subset X$ is called a \emph{(positive) cone} in $X$ {if (i) $X^+ \cap (-X^+) = \{0\}$, (ii) $\R_+ \cdot X^+ \subset X^+$, and (iii) $X^+ + X^+ \subset X^+$.} A cone $X^+$ introduces a partial order ``$\leq$'' on $X$ via%
\begin{equation*}
  x \leq y \quad \text{whenever} \quad y - x \in X^+.%
\end{equation*}

The pair $(X,X^+)$ is also called an \emph{ordered vector space}. If $X$ is a Banach space and the cone $X^+$ is closed, we call $(X,X^+)$ an \emph{ordered Banach space}. In this case, the cone $X^+$ is called \emph{generating} if $X^+ + (-X^+) = X$.
Clearly, a cone $X^+$ is generating if and only if $X^+$ spans $X$. {If the cone $X^+$ is generating}, then by \cite[Thm.~2.37]{AlT07} there exists a constant $M > 0$ such that every $x \in X$ can be decomposed as
\begin{align}
	\label{eq_bounded_decomposition}
	x = y-z \qquad \text{where} \quad y,z \ge 0 \quad \text{and} \quad \|y\|_X,\|z\|_X \le M \|x\|_X.
\end{align} 

{
The norm in $X$ is called \emph{monotone} if for any $x_1,x_2 \in X$ with $0 \leq x_1 \leq x_2$ it holds that $\|x_1\|_X \leq \|x_2\|_X$.
The cone $X^+$ is called \emph{normal} if there exists $\delta>0$ so that for any $x_1,x_2 \in X$ with $0 \leq x_1 \leq x_2$ it holds that $\|x_1\|_X \leq \delta \|x_2\|_X$. 
In this case, one can always find an equivalent norm which is monotone \cite[Thm.~2.38]{AlT07}.%
}


Let $(X,X^+)$ and $(Y,Y^+)$ be ordered vector spaces. We say that a map $f:X^+ \rightarrow Y^+$ is \emph{a (nonlinear) monotone operator} if $x_1 \leq x_2$ implies $f(x_1) \leq f(x_2)$ for all $x_1,x_2 \in X^+$.%

\section{Control systems and their stability}

In this paper, we work with the following definition of a control system (which provides all the features that are necessary for a global stability analysis).%

\begin{definition}\label{def_controlsystem}
Consider a triple $\Sigma = (X,\Uc,\phi)$ consisting of the following:%
\begin{enumerate}  
\item[(i)] A normed vector space $(X,\|\cdot\|_X)$, called the \emph{state space}.%
\item[(ii)] A vector space $U$ of \emph{input values} and a normed vector \emph{space of inputs} $(\Uc,\|\cdot\|_{\Uc})$, where $\Uc$ is a linear subspace of $U^{\R_+}$. We assume that the following  axioms hold:%
\begin{itemize}
\item \emph{The axiom of shift invariance}: for all $u \in \Uc$ and all $\tau\geq0$, the time-shifted function $u(\cdot + \tau)$ belongs to $\Uc$ with \mbox{$\|u\|_\Uc \geq \|u(\cdot + \tau)\|_\Uc$}.%
\item \emph{The axiom of concatenation}: for all $u_1,u_2 \in \Uc$ and for all $t>0$ the concatenation of $u_1$ and $u_2$ at time $t$, defined by%
\begin{equation*}
	\ccat{u_1}{u_2}{t}(\tau) :=
	\begin{cases}
		u_1(\tau)	& {\text{ for all } \tau \in [0,t],} \\ 
		u_2(\tau-t)	& {\text{ for all } \tau \in (t, \infty)}
	\end{cases}
\end{equation*}
belongs to $\Uc$.%
\end{itemize}
\item[(iii)] A map $\phi:D_{\phi} \to X$, $D_{\phi}\subseteq \R_+ \tm X \tm \Uc$, called \emph{transition map}, so that for all $(x,u)\in X \tm \Uc$ it holds that $D_{\phi} \cap (\R_+ \tm \{(x,u)\}) = [0,t_m)\tm \{(x,u)\}$, for a certain $t_m = t_m(x,u)\in (0,+\infty]$. The corresponding interval $[0,t_m)$ is called the \emph{maximal domain of definition} of the mapping $t\mapsto \phi(t,x,u)$, which we call a \emph{trajectory} of the system.%
\end{enumerate}
The triple $\Sigma$ is called a \emph{(control) system} if it satisfies the following axioms:%
\begin{enumerate}
\item[($\Sigma{1}$)]\label{axiom:Identity} \emph{The identity property:} for all $(x,u) \in X \tm \Uc$, it holds that $\phi(0,x,u) = x$.%
\item[($\Sigma{2}$)]\emph{Causality:} for all $(t,x,u) \in D_\phi$ and $\tilde{u} \in \Uc$ such that $u(s) = \tilde{u}(s)$ for all $s \in [0,t]$, it holds that $[0,t]\tm \{(x,\tilde{u})\} \subset D_\phi$ and $\phi(t,x,u) = \phi(t,x,\tilde{u})$.%
\item[($\Sigma{3}$)]\label{axiom:Continuity} \emph{Continuity:} for each $(x,u) \in X \tm \Uc$, the trajectory $t \mapsto \phi(t,x,u)$ is continuous on its maximal domain of definition.%
\item[($\Sigma{4}$)]\label{axiom:Cocycle} \emph{The cocycle property:} for all $x \in X$, $u \in \Uc$ and $t,h \geq 0$ so that $[0,t+h]\tm \{(x,u)\} \subset D_{\phi}$, we have $\phi(h,\phi(t,x,u),u(t+\cdot)) = \phi(t+h,x,u)$. 
\hfill $\lhd$
\end{enumerate}
\end{definition}

This class of systems encompasses control systems generated by ordinary differential equations, switched systems, time-delay systems,
many classes of partial differential equations, important classes of boundary control systems and many other systems.%

\begin{definition}\label{def_forward_completeness} 
We say that a control system $\Sigma = (X,\Uc,\phi)$ is \emph{forward complete} if $D_\phi = \R_+ \tm X \tm \Uc$, i.e., $\phi(t,x,u)$ is defined for all $(t,x,u) \in \R_+ \tm X \tm \Uc$.
\hfill $\lhd$
\end{definition}

An important property of ordinary differential equations with Lipschitz continuous right-hand sides is the possibility of extending a solution, which is bounded on a time interval $[0,t)$, to a larger time interval $[0,t+\ep)$. Evolution equations in Banach spaces with bounded control operators and Lipschitz continuous right-hand sides have similar properties \cite[Thm.~4.3.4]{CaH98}; the same holds for many other classes of systems \cite[Ch.~1]{KaJ11b}. The next property, adopted from \cite[Def.~1.4]{KaJ11b}, formalizes this behavior for general control systems.%

\begin{definition}\label{def_BIC} 
We say that a system $\Sigma$ satisfies the \emph{boundedness-implies-continuation (BIC) property} if for each $(x,u)\in X \tm \Uc$ such that the maximal existence time $t_m = t_m(x,u)$ is finite, for any given $M>0$ there exists $t \in [0,t_m)$ with $\|\phi(t,x,u)\|_X > M$.
\hfill $\lhd$
\end{definition}


Next, we introduce the input-to-state stability property, which unifies the classical asymptotic stability concept with the input-output stability notion, and is one of the cornerstones of nonlinear control theory \cite{KoA01,Son08}.

\begin{definition}\label{def_ISS}
	A system $\Sigma = (X,\Uc,\phi)$ is called \emph{(uniformly) input-to-state stable (ISS)} if there exist $\beta \in \KL$ and 
	{$\gamma \in \K \cup \{0\}$} 
such that 
\begin {equation*}
  \|\phi(t,x,u)\|_X \leq \beta(\|x\|_X,t) + \gamma(\|u\|_{\Uc}),\quad (t,x,u) \in D_{\phi}.
\end{equation*}
\end{definition}

Two properties, implied by ISS, will be important in the sequel:%
\begin{definition}
A system $\Sigma = (X,\Uc,\phi)$ is called \emph{uniformly globally stable (UGS)} if there exist $\sigma \in\Kinf$ and {$\gamma \in \K \cup \{0\}$} such that 
\begin{equation*}
  \|\phi(t,x,u)\|_X \leq \sigma(\|x\|_X) + \gamma(\|u\|_{\Uc}),\quad (t,x,u) \in D_{\phi}.
\end{equation*}
\end{definition}

\begin{definition}
A forward complete system $\Sigma = (X,\Uc,\phi)$ has the \emph{bounded input uniform asymptotic gain (bUAG) property} if there exists a {$\gamma \in \K \cup \{0\}$} such that for all $ \ep,r>0$ there is a time $\tau = \tau(\ep,r) \geq 0$ for which
\begin{equation*}
 \|x\|_X\leq r \ \wedge \ \|u\|_{\Uc} \leq r \ \wedge \ t \geq \tau\ \quad \Rightarrow \quad \|\phi(t,x,u)\|_X \leq \ep + \gamma(\|u\|_{\Uc}).%
\end{equation*}
\end{definition}

The UGS and bUAG properties are extensions of global Lyapunov stability and uniform global attractivity to systems with inputs.%

The following lemma provides a useful criterion for the input-to-state stability in terms of uniform global stability and the bUAG property (see \cite[Lem.~3.7]{Mir19b}). It is a special case of stronger ISS characterizations shown in \cite{MiW18b} and \cite[Sec.~6]{Mir19b}.%

{
\begin{lemma}\label{lem_UGS_and_bUAG_imply_UAG}
Let $\Sigma = (X,\Uc,\phi)$ be a control system with the BIC property. If $\Sigma$ is UGS and has the bUAG property\footnote{Note that UGS combined with the BIC property implies forward completeness, thus the bUAG property makes sense here.}, then $\Sigma$ is ISS.
\end{lemma}
}

\section{Infinite interconnections}
\label{sec:Infinite interconnections}

In this subsection, we introduce (feedback) interconnections of an arbitrary number of control systems, indexed by some nonempty set $I$. For each $i \in I$, let $(X_i,\|\cdot\|_{X_i})$ be a normed vector space which will serve as the state space of a control system $\Sigma_i$. Before we can specify the space of inputs for $\Sigma_i$, we first have to construct the overall state space. In the following, we use the sequence notation $(x_i)_{i \in I}$ for functions with domain $I$. The overall state space is then defined as%
\begin{equation*}
 	X := \Bigl\{ (x_i)_{i \in I} 
 	{\in \prod_{i \in I} X_i : } 
 	\, \sup_{i \in I} \|x_i\|_{X_i} < \infty \Bigr\}.%
\end{equation*}
{It is a vector space with respect to pointwise addition and scalar multiplication, 
and we can turn it into a normed space in the following way:}%

\begin{proposition}
	{The state space $X$ is a normed space with respect to the norm
	\begin{equation*}
 		\|x\|_X := \sup_{i\in I}\|x_i\|_{X_i}.%
	\end{equation*}
	If all of the spaces $(X_i,\|\cdot\|_{X_i})$ are Banach spaces, then so is $(X,\|\cdot\|_X)$.}
\end{proposition}


{The proof of the proposition is straightforward, hence we omit it.}

We also define for each $i \in I$ the normed vector space $X_{\neq i}$ by the same construction as above, but for the restricted index set $I \setminus \{i\}$. Then $X_{\neq i}$ can be identified with the closed linear subspace $\{ (x_j)_{j \in I} \in X : x_i = 0 \}$ of $X$.%

Now consider for each $i \in I$ a control system of the form%
\begin{equation*}
  \Sigma_i = (X_i,\PC_b(\R_+,X_{\neq i}) \tm \Uc,\bar{\phi}_i),%
\end{equation*}
where $\PC_b(\R_+,X_{\neq i})$ is the space of all globally bounded piecewise continuous functions $w:\R_+ \rightarrow X_{\neq i}$, with the norm $\|w\|_{\infty} = \sup_{t \geq 0}\|w(t)\|_{X_{\neq i}}$. The norm on $\PC_b(\R_+,X_{\neq i}) \tm \Uc$ is defined by%
\begin{equation}\label{eq_product_input_norm}
  \|(w,u)\|_{\PC_b(\R_+,X_{\neq i}) \tm \Uc} := \max\left\{ \|w\|_{\infty}, \|u\|_{\Uc} \right\}.%
\end{equation}
Here we assume that $\Uc \subset U^{\R_+}$ for some vector space $U$, and $\Uc$ satisfies the axioms of shift invariance and concatenation. Then, by the definition of $\PC_b(\R_+,X_{\neq i})$ and the norm \eqref{eq_product_input_norm}, these axioms are also satisfied for the product space $\PC_b(\R_+,X_{\neq i}) \tm \Uc$.%

\begin{definition}\label{def_interconnection}
Given the control systems $\Sigma_i$ ($i \in I$) as above, assume that there is a map $\phi:D_\phi \to X$, defined on $D_\phi \subset \R_+ \tm X \tm \Uc$, such that:
\begin{enumerate}
	\item[(i)] For each $x \in X$ and each $u \in \Uc$ there is $\varepsilon>0$ such that $[0,\varepsilon] \tm \{(x,u)\} \subset D_\phi$.
	\item[(ii)] Furthermore, 
the components $\phi_i$ of the transition map $\phi:D_{\phi} \rightarrow X$ satisfy%
\begin{equation*}
  \phi_i(t,x,u) = \bar{\phi}_i(t,x_i,(\phi_{\neq i},u)) \mbox{\quad for all\ } (t,x,u) \in D_{\phi},%
\end{equation*}
where $\phi_{\neq i}(\cdot) = (\phi_j(\cdot,x,u))_{j \in I \setminus \{i\}}$ for all $i \in I$.\footnote{By the causality axiom, we can assume that $\phi_{\neq i}$ is globally bounded, since $\bar{\phi}_i(t,x_i,(\phi_{\neq i},u))$ does not depend on the values $\phi_{\neq i}(s)$ with $s > t$, and on the compact interval $[0,t]$, $\phi_{\neq i}$ is bounded because it is continuous.}

We also assume that $\phi$ is maximal in the sense that no other map $\tilde\phi:\tilde{D}_\phi \to X$ with $\tilde{D}_\phi \supset D_\phi$ exists, which satisfies all of the above properties, and coincides with $\phi$ on $D_\phi$.%

\end{enumerate}

If the map $\phi$ is unique with above properties, and if $\Sigma = (X,\Uc,\phi)$ is a control system satisfying BIC property, then $\Sigma$ is called the \emph{(feedback) interconnection} of the systems $\Sigma_i$.

We then call $X_{\neq i}$ the space of \emph{internal input values}, $\PC_b(\R_+,X_{\neq i})$ the space of \emph{internal inputs}, and $\Uc$ the space of \emph{external inputs} of the system $\Sigma_i$. Moreover, we call $\Sigma_i$ the \emph{$i$-th subsystem} of $\Sigma$.
\hfill $\lhd$
\end{definition}


The stability properties introduced {above} are defined in terms of the {norm} of the whole input, and this is not suitable for the consideration of coupled systems, as we are interested not only in the collective influence of all inputs on a subsystem, but in the influence of particular subsystems on a given subsystem. The next definition provides the needed flexibility.
%

\begin{definition}\label{def_subsys_iss_semimax}
Given the spaces $(X_j,\|\cdot\|_{X_j})$, $j\in I$, and the system $\Sigma_i$ for a fixed $i \in I$, we say that $\Sigma_i$ is \emph{input-to-state stable (ISS) (in semi-maximum formulation)} if $\Sigma_i$ {is} forward complete and there are $\gamma_{ij},\gamma_j \in \K \cup \{0\}$ for all $j \in I$, and $\beta_i \in \KL$ such that for all initial states $x_i \in X_i$, all internal inputs $w_{\neq i} = (w_j)_{j\in I \setminus \{i\}} \in \PC_b(\R_+,X_{\neq i})$, all external inputs $u \in \Uc$ and $t \geq 0$:%
\begin{equation*}
  \|\bar{\phi}_i(t,x_i,(w_{\neq i},u))\|_{X_i} \leq \beta_i(\|x_i\|_{X_i},t) + \sup_{j \in I}\gamma_{ij}(\|w_j\|_{[0,t]}) + \gamma_i(\|u\|_{\Uc}).%
\end{equation*}
Here we assume that the functions $\gamma_{ij}$ satisfy $\sup_{j \in I}\gamma_{ij}(r) < \infty$ for every $r \geq 0$ (implying that the sum on the 
{right-hand side} 
is finite) and $\gamma_{ii} = 0$.
\hfill $\lhd$
\end{definition}

The functions $\gamma_{ij}$ and $\gamma_i$ in this definition are called \emph{(nonlinear) gains}. 

Assuming that all systems $\Sigma_i$, $i\in I$, are ISS in semi-maximum formulation, we can define a nonlinear monotone operator $\Gamma_{\otimes}:\ell_{\infty}(I)^+ \rightarrow \ell_{\infty}(I)^+$ from the gains $\gamma_{ij}$ {by}
\begin{equation}
\label{eq:Gain-operator-semimax}
  \Gamma_{\otimes}(s) := \bigl(\sup_{j\in I}\gamma_{ij}(s_j)\bigr)_{i\in I},\quad s = (s_i)_{i\in I} \in \ell_{\infty}(I)^+.%
\end{equation}

In general, $\Gamma_{\otimes}$ is not well-defined. It is easy to see that the following assumption is equivalent to $\Gamma_{\otimes}$ being well-defined.%

{
\begin{assumption}\label{ass_gammamax_welldef}
For every $r>0$, we have $\sup_{i,j \in I}\gamma_{ij}(r) < \infty$.
\end{assumption}
}

\begin{lemma}
Assumption \ref{ass_gammamax_welldef} is equivalent to the existence of $\zeta\in\Kinf$ and $a\geq0$ such that $\sup_{i,j\in I}\gamma_{ij}(r) \leq a + \zeta(r)$ for all $r\geq 0$.%
\end{lemma}

\begin{proof}
Obviously, the implication ``$\Leftarrow$'' holds. Conversely, define $\xi: \R_+ \to \R_+$ by%
\begin{equation*}
  \xi(r) := \sup_{i,j\in I}\gamma_{ij}(r).%
\end{equation*}
As a supremum of continuous increasing functions, $\xi$ is lower semicontinuous and nondecreasing on its domain of definition. As $\xi(r)$ is finite for every $r\geq 0$ by assumption, define%
\begin{equation*}
\tilde{\xi}(r):=
\begin{cases}
0 & \mbox{if } r=0,\\
\xi(r)-a & \mbox{if } r>0,%
\end{cases}
\end{equation*}
where $a:=\lim_{r\to+0}\xi(r)\geq 0$ (the limit exists as $\xi$ is nondecreasing). By construction, $\tilde{\xi}$ is nondecreasing, continuous at $0$ and satisfies $\tilde{\xi}(0)=0$. Hence, $\tilde{\xi}$ can be upper bounded by a certain $\zeta\in \Kinf$ (this follows from a more general result in \cite[Prop.~9]{MiW19b}). Overall, $\sup_{i,j\in I} \gamma_{ij}(r) \leq a + \zeta(r)$ for all $r\geq 0$. \qed%
\end{proof}

Also observe that $\Gamma_{\otimes}$, if well-defined, is a monotone operator:%
\begin{equation*}
  s^1 \leq s^2 \quad\Rightarrow\quad \Gamma_{\otimes}(s^1) \leq \Gamma_{\otimes}(s^2) \mbox{\quad for all\ } s^1,s^2 \in \ell_{\infty}(I)^+.%
\end{equation*}

\begin{remark}\label{rem_lg_sg_operator}
If all gains $\gamma_{ij}$ are linear, then $\Gamma_{\otimes}$ satisfies the following two properties:%
\begin{itemize}
\item $\Gamma_{\otimes}$ is a homogeneous operator of degree one, i.e., $\Gamma_{\otimes}(as) = a\Gamma_{\otimes}(s)$ for all $a \geq 0$ and $s \in \ell_{\infty}(I)^+$.%
\item $\Gamma_{\otimes}$ is subadditive, i.e., $\Gamma_{\otimes}(s^1 + s^2) \leq \Gamma_{\otimes}(s^1) + \Gamma_{\otimes}(s^2)$ for all $s^1,s^2 \in \ell_{\infty}(I)^+$.%
\end{itemize}
\end{remark}

Finally, we provide a criterion for continuity of $\Gamma_{\otimes}$ (this criterion with a slightly different statement can already be found in \cite[Lem.~2.1]{DMS19a}, however, without proof).%

\begin{proposition}
Assume that the family $\{\gamma_{ij}\}_{(i,j) \in I^2}$ is pointwise equicontinuous, i.e., for every $r \in \R_+$ and $\ep>0$ there exists $\delta>0$ such that $|\gamma_{ij}(r) - \gamma_{ij}(s)| \leq \ep$ whenever $(i,j) \in I^2$ and $|r - s| \leq \delta$. Then $\Gamma_{\otimes}$ is well-defined and continuous.%
\end{proposition}

\begin{proof}
First, we show that $\Gamma_{\otimes}$ is well-defined. Fixing some $r>0$, the family $\{\gamma_{ij}\}_{(i,j)\in I^2}$ is uniformly equicontinuous on the compact interval $[0,r]$, which follows by a compactness argument. Hence, we can find $\delta>0$ so that $|s_1 - s_2| \leq \delta$ with $s_1,s_2 \in [0,r]$ implies $|\gamma_{ij}(s_1) - \gamma_{ij}(s_2)| \leq 1$ for all $i,j$. We can assume that $\delta$ is of the form $r/n$ for an integer $n$. Then%
\begin{equation*}
  \gamma_{ij}(r) = \sum_{k=0}^{n-1} \Bigl[\gamma_{ij}\Bigl(\frac{k+1}{n}r\Bigr) - \gamma_{ij}\Bigl(\frac{k}{n}r\Bigr)\Bigr] \leq n < \infty%
\end{equation*}
for all $(i,j) \in I^2$. Hence, $\Gamma_{\otimes}$ is well-defined.

Now we prove continuity. Choose any $\varepsilon>0$, fix some $s^0 \in \ell_{\infty}(I)^+$ and let $s \in \ell_{\infty}(I)^+$ so that $\|s - s^0\|_{\ell_{\infty}(I)} \leq \delta$ for some $\delta>0$ to be determined. By the required equicontinuity, we can choose $\delta$ small enough so that $|\gamma_{ij}(s^0_j) - \gamma_{ij}(s_j)| \leq \ep$ for all $(i,j)$ as $|s^0_j - s_j| \leq \|s^0 - s\|_{\ell_{\infty}(I)} \leq \delta$. This also implies%
\begin{equation*}
  \|\Gamma_{\otimes}(s^0) - \Gamma_{\otimes}(s)\|_{\ell_{\infty}(I)} = \sup_{i \in I}\Bigl|\sup_{j \in I} \gamma_{ij}(s^0_j) - \sup_{j \in I} \gamma_{ij}(s_j)\Bigr| \leq \ep.%
\end{equation*}
{In the last inequality, we use the estimate%
\begin{equation*}
  \sup_{j \in I} \gamma_{ij}(s^0_j) - \sup_{j \in I} \gamma_{ij}(s_j) \leq \sup_{j \in I} (\gamma_{ij}(s_j) + \ep) - \sup_{j \in I} \gamma_{ij}(s_j) = \ep,%
\end{equation*}
and the analogous estimate in the other direction.} \qed
\end{proof}

Another formulation of ISS for the systems $\Sigma_i$ is as follows. In this formulation, we need to assume that $I$ is countable.%

\begin{definition}\label{def_subsys_iss_sum}
Assume that $I$ is a nonempty countable set. Given the spaces $(X_j,\|\cdot\|_{X_j})$, $j\in I$, and the system $\Sigma_i$ for a fixed $i \in I$, we say that $\Sigma_i$ is \emph{input-to-state stable (ISS) (in summation formulation)} if $\Sigma_i$ is forward complete and there are $\gamma_{ij},\gamma_j \in \K \cup \{0\}$ for all $j \in I$, and $\beta_i \in \KL$ such that for all initial states $x_i \in X_i$, all internal inputs $w_{\neq i} = (w_j)_{j\in I \setminus \{i\}} \in \PC_b(\R_+,X_{\neq i})$, all external inputs $u \in \Uc$ and $t \geq 0$:%
\begin{equation*}
  \|\bar{\phi}_i(t,x_i,(w_{\neq i},u))\|_{X_i} \leq \beta_i(\|x_i\|_{X_i},t) + \sum_{j \in I}\gamma_{ij}(\|w_j\|_{[0,t]}) + \gamma_i(\|u\|_{\Uc}).%
\end{equation*}
Here we assume that the functions $\gamma_{ij}$ are such that $\sum_{j \in I}\gamma_{ij}(r) < \infty$ for every $r \geq 0$ (implying that the sum on the 
{right-hand side} 
is finite) and $\gamma_{ii} = 0$.
\hfill $\lhd$
\end{definition}

{
\begin{remark}
\label{rem:Why-many-formulations} 
If a network has finitely many components, ISS in summation formulation, and ISS in semi-maximum formulation are equivalent concepts. Nevertheless, even for finite networks the gains in semi-maximum formulation and the gains in summation formulation are distinct, and for some systems one formulation is better than the other one in the sense that it produces tighter (and thus smaller) gains. This motivates the interest in analyzing both formulations. We illustrate this by examples in Sections~\ref{examp: a linear spatially invariant system}, \ref{examp: a nonlinear spatially invariant system}. In fact, also more general formulations of input-to-state stability for networks are studied in the literature \cite{DRW10}, using the formalism of monotone aggregation functions.
\end{remark}
}

Assuming that all systems $\Sigma_i$, $i\in I$, are ISS, we can define a nonlinear monotone operator $\Gamma_{\boxplus}:\ell_{\infty}(I)^+ \rightarrow \ell_{\infty}(I)^+$ from the gains $\gamma_{ij}$ as follows:%
\begin{equation*}
  \Gamma_{\boxplus}(s) := \Bigl(\sum_{j\in I}\gamma_{ij}(s_j)\Bigr)_{i \in I},\quad s = (s_i)_{i\in I} \in \ell_{\infty}(I)^+.%
\end{equation*}
Again, $\Gamma_{\boxplus}$ might not be well-defined, hence we need to make an appropriate assumption.%

\begin{assumption}\label{ass_gammasum_welldef}
For every $r > 0$, we have%
\begin{equation*}
  \sup_{i \in I}\sum_{j \in I} \gamma_{ij}(r) < \infty.%
\end{equation*}
\end{assumption}

\begin{remark}
Assume that all the gains $\gamma_{ij}$, $(i,j) \in I^2$, are linear functions. Then the gain operator $\Gamma_{\boxplus}$ can be regarded as a linear operator on $\ell_{\infty}(I)$ and Assumption \ref{ass_gammasum_welldef} {is equivalent to $\Gamma_{\boxplus}$ being a bounded} linear operator on $\ell_{\infty}(I)$.
\end{remark}

\begin{proposition}
Assume that the operator $\Gamma_{\boxplus}$ is well-defined. A sufficient criterion for continuity of $\Gamma_{\boxplus}$ is that each $\gamma_{ij}$ is a $C^1$-function and%
\begin{equation*}
  \sup_{i \in I} \sum_{j \in I} \sup_{0 < s \leq r} \gamma_{ij}'(s) < \infty \mbox{\quad for all\ } r > 0.%
\end{equation*}
\end{proposition}


\begin{proof}
Fix $s^0 = (s^0_j)_{j\in I} \in \ell_{\infty}(I)^+$ and $\ep>0$. Let $s \in \ell_{\infty}(I)^+$ with $\|s - s^0\|_{\ell_{\infty}(I)} = \sup_{i \in I}|s_i - s^0_i| \leq \delta$ for some $\delta > 0$, to be determined later. Then%
\begin{equation*}
  \|\Gamma_{\boxplus}(s^0) - \Gamma_{\boxplus}(s)\|_{\ell_{\infty}(I)} = \sup_{i \in I}\Bigl| \sum_{j \in I} (\gamma_{ij}(s^0_j) - \gamma_{ij}(s_j)) \Bigr|.%
\end{equation*}
Using the assumption that each $\gamma_{ij}$ is a $C^1$-function and writing $s^0_{\max} := \|s^0\|_{\ell_{\infty}(I)}$, we can estimate this by%
\begin{align*}
  & \|\Gamma_{\boxplus}(s) - \Gamma_{\boxplus}(s^0)\|_{\ell_{\infty}(I)} \leq \sup_{i \in I} \sum_{j \in I} |\gamma_{ij}(s_j) - \gamma_{ij}(s^0_j)| \\
	                                   &\leq \sup_{i \in I} \sum_{j \in I} \sup_{r \in [s^0_j - \delta,s^0_j + \delta]}|\gamma_{ij}'(r)| |s_j - s^0_j|
	                                   \leq \delta \sup_{i \in I} \sum_{j \in I} \sup_{r \leq s^0_{\max}+\delta} \gamma_{ij}'(r).%
\end{align*}
By assumption, the last supremum is finite, which implies that $\delta$ can be chosen small enough so that the whole expression is smaller than $\ep$. \qed%
\end{proof}

We also need versions of UGS for the systems $\Sigma_i$.%

\begin{definition}\label{def_subsys_ugs_semimax}
Given the spaces $(X_j,\|\cdot\|_{X_j})$, $j\in I$, and the system $\Sigma_i$ for a fixed $i \in I$, we say that $\Sigma_i$ is \emph{uniformly globally stable (UGS) (in semi-maximum formulation)} if $\Sigma_i$ is forward complete and there are $\gamma_{ij},\gamma_j \in \K \cup \{0\}$ for all $j \in I$, and $\sigma_i \in \Kinf$ such that for all initial states $x_i \in X_i$, all internal inputs $w_{\neq i} = (w_j)_{j\in I \setminus \{i\}} \in \PC_b(\R_+,X_{\neq i})$, all external inputs $u \in \Uc$ and $t \geq 0$:%
\begin{equation*}
  \|\bar{\phi}_i(t,x_i,(w_{\neq i},u))\|_{X_i} \leq \sigma_i(\|x_i\|_{X_i}) + \sup_{j \in I}\gamma_{ij}(\|w_j\|_{[0,t]}) + \gamma_i(\|u\|_{\Uc}).%
\end{equation*}
Here we assume that the functions $\gamma_{ij}$ are such that $\sup_{j \in I}\gamma_{ij}(r) < \infty$ for every $r \geq 0$ (implying that the sum on the 
{right-hand side} 
is finite) and $\gamma_{ii} = 0$.%
\hfill $\lhd$\end{definition}

\begin{definition}\label{def_subsys_ugs_sum}
Let $I$ be a countable index set. Given the spaces $(X_j,\|\cdot\|_{X_j})$, $j\in I$, and the system $\Sigma_i$ for a fixed $i \in I$, we say that $\Sigma_i$ is \emph{uniformly globally stable (UGS) (in summation formulation)} if $\Sigma_i$ is forward complete and there are $\gamma_{ij},\gamma_j \in \K \cup \{0\}$ for all $j \in I$, and $\sigma_i \in \Kinf$ such that for all initial states $x_i \in X_i$, all internal inputs $w_{\neq i} = (w_j)_{j\in I \setminus \{i\}} \in \PC_b(\R_+,X_{\neq i})$, all external inputs $u \in \Uc$ and $t \geq 0$:%
\begin{equation*}
  \|\bar{\phi}_i(t,x_i,(w_{\neq i},u))\|_{X_i} \leq \sigma_i(\|x_i\|_{X_i}) + \sum_{j \in I}\gamma_{ij}(\|w_j\|_{[0,t]}) + \gamma_i(\|u\|_{\Uc}).%
\end{equation*}
Here we assume that the functions $\gamma_{ij}$ are such that $\sum_{j \in I}\gamma_{ij}(r) < \infty$ for every $r \geq 0$ (implying that the sum on the 
{right-hand side} 
is finite) and $\gamma_{ii} = 0$.%
\hfill $\lhd$\end{definition}

\section{Stability of discrete-time systems}\label{sec:Stability of discrete-time systems}

{
In this section, we study stability properties of the system%
\begin{equation}
	\label{eq_monotone_system}
 	x(k+1) \leq A(x(k)) + u(k),\quad k \in \Z_+.
\end{equation}
Here, $(X,X^+)$ is an ordered Banach space, $A: X^+ \rightarrow X^+$ is a nonlinear operator on the cone $X^+$, and the input $u$ is an element of $\ell_{\infty}(\Z_+,X^+)$, where the latter space is defined as%
\begin{equation*}
  \ell_{\infty}(\Z_+,X^+) := \{u = (u(k))_{k\in\Z_+} : u(k) \in X^+,\ \|u\|_{\infty} := \sup_{k\in\Z_+}\|u(k)\|_X < \infty\}.%
\end{equation*}
A \emph{solution} of the equation~\eqref{eq_monotone_system} is a mapping $x: \Z_+ \to X^+$ that satisfies~\eqref{eq_monotone_system}.
We call a mapping $x: \Z_+ \to X^+$ \emph{decreasing} if $x(k+1) \le x(k)$ for all $k \in \Z_+$.
}

As we will see, for the small-gain analysis of infinite interconnections, the properties of the gain operator and the discrete-time system \eqref{eq_monotone_system} induced by the gain operator, are essential. {So we now relate the stability of the system \eqref{eq_monotone_system} to the properties of the operator $A$.}%

{
\begin{definition}
The system \eqref{eq_monotone_system} has the \emph{monotone limit property (MLIM)} if there is $\xi \in \Kinf$ such that for every $\ep>0$, every constant input $u(\cdot) :\equiv w \in X^+$ and every decreasing solution $x: \Z_+ \to X^+$ of \eqref{eq_monotone_system},
there exists $N = N(\ep,u,x(\cdot)) \in \Z_+$ with%
\begin{equation*}
  \|x(N)\|_X \leq \ep + \xi(\|w\|_X).%
\end{equation*}
\end{definition}
}

\begin{definition}
Let $(X,X^+)$ be an ordered Banach space and let $A:X^+ \rightarrow X^+$ be a nonlinear operator. We say that $\id - A$ has the \emph{monotone bounded invertibility (MBI) property} if there exists $\xi \in \Kinf$ such that {for all $v,w \in X^+$ the following implication holds:}%
\begin{equation*}
  (\id - A)(v) \leq w \quad \Rightarrow \quad \|v\|_X \leq \xi(\|w\|_X).%
\end{equation*}
\end{definition}

\begin{proposition}\label{prop_mlim_implies_mbip}
Let $(X,X^+)$ be an ordered Banach space and let $A:X^+ \rightarrow X^+$ be a nonlinear operator.
If system \eqref{eq_monotone_system} has the MLIM property, then the operator $\id - A$ has the MBI property.
\end{proposition}

\begin{proof}
Assume that $(\id - A)(v) \leq w$ for some $v,w \in X^+$. We write this as $v \leq A(v) + w$. Hence, $x(\cdot) :\equiv v$ is a constant solution of \eqref{eq_monotone_system} corresponding to the constant input sequence $u(\cdot) :\equiv w$. By the MLIM property, there exists $\xi \in \Kinf$ (independent of $v,w$) so that for every $\ep>0$ there is $N$ with%
\begin{equation*}
  \|v\|_X = \|x(N)\|_X \leq \ep + \xi(\|u\|_{\infty}) = \ep + \xi(\|w\|_X).%
\end{equation*}
Since this holds for every $\ep>0$, we obtain $\|v\|_X \leq \xi(\|w\|_X)$, which completes the proof. \qed%
\end{proof}

\emph{Whether the MBI property is strictly weaker than the MLIM property, or whether they are equivalent, is an open problem.} {In the next proposition, though, we show that they are equivalent under certain assumptions on the operator $A$ or the cone $X^+$.} Later, in Propositions~\ref{prop:eISS-criterion-linear-systems} and \ref{prop:join-morphism-eISS-criterion}, we show their equivalence for linear operators and for the gain operator $\Gamma_\otimes$ with linear gains, defined on $\ell_{\infty}(I)$.%

{The cone $X^+$ is said to have the \emph{Levi property} if every decreasing sequence in $X^+$ is norm-convergent \cite[Def.~2.44(2)]{AlT07}. Typical examples are the standard cones in $L^p$-spaces for $p \in [1,\infty)$, and the standard cone in the space $c_0$ of real-valued sequences that converge to $0$. We note in passing that if the cone $X^+$ has the Levi property, then it is normal \cite[Thm.~2.45]{AlT07}.}

{
\begin{proposition}\label{prop_compact_operators}
Let $(X,X^+)$ be an ordered Banach space with normal cone and let $A:X^+ \rightarrow X^+$ be a nonlinear, continuous and monotone operator. 
If the cone $X^+$ has the Levi property or if the operator $A$ is compact (i.e., it maps bounded sets onto precompact sets), then the following statements are equivalent:%
\begin{enumerate}
\item[(i)] System \eqref{eq_monotone_system} satisfies the MLIM property.%
\item[(ii)] The operator $\id - A$ satisfies the MBI property.%
\end{enumerate}
\end{proposition}

\begin{proof}
In view of Proposition \ref{prop_mlim_implies_mbip}, it suffices to prove the implication ``(ii) $\Rightarrow$ (i)''. Hence, consider a constant input $u(\cdot) :\equiv w \in X^+$ and a decreasing sequence $x(\cdot)$ in $X^+$ such that $x(k+1) \leq A(x(k)) + w$ for all $k \in \Z_+$. As $A$ is monotone, the operator $\tilde{A}(x) := A(x) + w$, $\tilde{A}:X^+ \rightarrow X^+$, is monotone as well; if $A$ is compact, then so is $\tilde A$. Moreover,%
\begin{equation}\label{eq_newineq}
  x(k+1) \leq \tilde{A}(x(k)) \mbox{\quad for all\ } k \in \Z_+.%
\end{equation}
Now consider the sequence $y(k) := \tilde{A}(x(k))$, $k \in \Z_+$. As $x$ is decreasing, so is $y$. Next, we note that $y$ converges in norm. Indeed, if the cone has the Levi property, this is clear. If the cone does not have the Levi property, then $A$, and thus $\tilde{A}$, is compact by assumption. So $y$ has a convergent subsequence; since $y$ is decreasing and the cone is normal, it thus follows that $y$ converges itself.%

Let $y_* \in X^+$ denote the limit of the sequence $y$. Applying $\tilde{A}$ on both sides of \eqref{eq_newineq}, yields%
\begin{equation*}
  y(k+1) \leq \tilde{A}(y(k)) \mbox{\quad for all\ } k \in \Z_+.%
\end{equation*}
Taking the limit for $k \rightarrow \infty$ and using continuity of $A$ results in $y_* \leq \tilde{A}(y_*) = A(y_*) + w$. Since this can be written as $(\id - A)(y_*) \leq w$, the MBI property of $\id - A$ gives $\|y_*\| \leq \xi(\|w\|)$. As $X^+$ is a normal cone, there is $\delta>0$ such that for every $\varepsilon>0$ there is $k>0$ large enough, for which
\begin{equation*}
  \|x(k+1)\|_X \leq \delta\|\tilde{A}(x(k))\|_X \leq \ep + \delta\xi(\|w\|_X).%
\end{equation*}
This completes the proof. \qed
\end{proof}
}

\section{Small-gain theorems}\label{sec:Small-gain theorems}

\subsection{Small-gain theorems in semi-maximum formulation}

In this subsection, we prove small-gain theorems for UGS and ISS, both in semi-maximum formulation. We start with UGS.%

\begin{theorem}\label{thm_ugs_semimax_sg}{\textbf{(UGS small-gain theorem in semi-maximum formulation)}}
Let $I$ be an arbitrary nonempty index set, $(X_i,\|\cdot\|_{X_i})$, $i\in I$, normed spaces and $\Sigma_i = (X_i,\PC_b(\R_+,X_{\neq i}) \tm \Uc,\bar{\phi}_i)$ forward complete control systems. Assume that the interconnection $\Sigma = (X,\Uc,\phi)$ of the systems $\Sigma_i$ is well-defined. Furthermore, let the following assumptions be satisfied:%
\begin{enumerate}
\item[(i)] Each system $\Sigma_i$ is UGS in the sense of Definition \ref{def_subsys_ugs_semimax} with $\sigma_i \in \K$ and nonlinear gains $\gamma_{ij},\gamma_i \in \K \cup \{0\}$.%
\item[(ii)] There exist $\sigma_{\max} \in \Kinf$ and $\gamma_{\max} \in \Kinf$ so that $\sigma_i \leq \sigma_{\max}$ and $\gamma_i \leq \gamma_{\max}$, pointwise for all $i \in I$.%
\item[(iii)] Assumption \ref{ass_gammamax_welldef} is satisfied for the operator $\Gamma_{\otimes}$ defined via the gains $\gamma_{ij}$ from (i) and $\id - \Gamma_{\otimes}$ has the MBI property.%
\end{enumerate}
Then $\Sigma$ is forward complete and UGS.%
\end{theorem}

\begin{proof}
Fix $(t,x,u) \in D_{\phi}$ and observe that%
\begin{equation*}
  \|\phi(t,x,u)\|_X = \sup_{i \in I} \|\phi_i(t,x,u)\|_{X_i} = \sup_{i \in I} \|\bar{\phi}_i(t,x_i,(\phi_{\neq i},u))\|_{X_i}.%
\end{equation*}
Abbreviating $\bar{\phi}_j(\cdot) = \bar{\phi}_j(\cdot,x_j,(\phi_{\neq j},u))$ and using assumption (i), we can estimate%
\begin{equation}\label{eq_ugs_firstest}
  \sup_{s \in [0,t]}\|\bar{\phi}_i(s,x_i,(\phi_{\neq i},u))\|_{X_i} \leq \sigma_i(\|x_i\|_{X_i}) + \sup_{j \in I}\gamma_{ij}(\|\bar{\phi}_j\|_{[0,t]}) + \gamma_i(\|u\|_{\Uc}).%
\end{equation}
From the inequalities (using continuity of $s \mapsto \phi(s,x,u)$)%
\begin{equation*}
  0 \leq \sup_{s\in[0,t]} \|\bar{\phi}_i(s,x_i,(\phi_{\neq i},u))\|_{X_i} \leq \sup_{s\in[0,t]} \|\phi(s,x,u)\|_X < \infty \mbox{\quad for all\ } i \in I,%
\end{equation*}
it follows that%
\begin{equation*}
  \vec{\phi}_{\max}(t) := \Bigl( \sup_{s \in [0,t]}\|\bar{\phi}_i(s,x_i,(\phi_{\neq i},u))\|_{X_i} \Bigr)_{i \in I} \in \ell_{\infty}(I)^+.%
\end{equation*}
From Assumption (ii), it follows that also the vectors $\vec{\sigma}(x) := (\sigma_i(\|x_i\|_{X_i}))_{i\in I}$ and $\vec{\gamma}(u) := ( \gamma_i(\|u\|_{\Uc}) )_{i \in I}$ are contained in $\ell_{\infty}(I)^+$. Hence, we can write the inequalities \eqref{eq_ugs_firstest} in vectorized form as%
\begin{equation*}
  (\id - \Gamma_{\otimes})(\vec{\phi}_{\max}(t)) \leq \vec{\sigma}(x) + \vec{\gamma}(u).%
\end{equation*}
By Assumption (iii), this yields for some $\xi\in\Kinf$, independent of $x,u$:%
\begin{equation*}
  \|\vec{\phi}_{\max}(t)\|_{\ell_{\infty}(I)} \leq \xi ( \|\vec{\sigma}(x) + \vec{\gamma}(u) \|_{\ell_{\infty}(I)} ) \leq \xi( \|\vec{\sigma}(x)\|_{\ell_{\infty}(I)} + \|\vec{\gamma}(u)\|_{\ell_{\infty}(I)} ).%
\end{equation*}
Since $\xi(a + b) \leq \max\{\xi(2a),\xi(2b)\} \leq \xi(2a) + \xi(2b)$ for all $a,b\geq 0$, this implies%
\begin{eqnarray*}
  \|\vec{\phi}_{\max}(t)\|_{\ell_{\infty}(I)} 
	&\leq& \xi(2\|\vec{\sigma}(x)\|_{\ell_{\infty}(I)}) + \xi(2\|\vec{\gamma}(u)\|_{\ell_{\infty}(I)})\\
	&\leq& \xi(2\sigma_{\max}(\|x\|_X)) + \xi(2\gamma_{\max}(\|u\|_{\Uc})),
\end{eqnarray*}
and we conclude that%
\begin{eqnarray*}
  \|\phi(t,x,u)\|_X \leq \|\vec{\phi}_{\max}(t)\|_{\ell_{\infty}(I)} \leq \xi(2\sigma_{\max}(\|x\|_X)) + \xi(2\gamma_{\max}(\|u\|_{\Uc})),%
\end{eqnarray*}
which is a UGS estimate with $\sigma(r) := \xi(2\sigma_{\max}(r))$, $\gamma(r) := \xi(2\gamma_{\max}(r))$ for $\Sigma$ for all $(t,x,u)\in D_\phi$. Since $\Sigma$ has the BIC property by assumption, it follows that $\Sigma$ is forward complete and UGS. \qed%
\end{proof}

Now we are in position to state the ISS small-gain theorem.%

\begin{theorem}\label{thm_smallgain_iss_semimax}\textbf{(Nonlinear ISS small-gain theorem in semi-maximum formulation)}
Let $I$ be an arbitrary nonempty index set, $(X_i,\|\cdot\|_{X_i})$, $i\in I$, normed spaces and $\Sigma_i = (X_i,\PC_b(\R_+,X_{\neq i}) \tm \Uc,\bar{\phi}_i)$ forward complete control systems. Assume that the interconnection $\Sigma = (X,\Uc,\phi)$ of the systems $\Sigma_i$ is well-defined. Furthermore, let the following assumptions be satisfied:%
\begin{enumerate}
\item[(i)] Each system $\Sigma_i$ is ISS in the sense of Definition \ref{def_subsys_iss_semimax} with $\beta_i \in \KL$ and nonlinear gains $\gamma_{ij},\gamma_i \in \K \cup \{0\}$.%
\item[(ii)] There are $\beta_{\max} \in \KL$ and $\gamma_{\max} \in \K$ so that $\beta_i \leq \beta_{\max}$ and $\gamma_i \leq \gamma_{\max}$ pointwise for all $i \in I$.%
\item[(iii)] Assumption \ref{ass_gammamax_welldef} holds and the discrete-time system%
\begin{equation}\label{eq_sg_system}
  w(k+1) \leq \Gamma_{\otimes}(w(k)) + v(k),%
\end{equation}
with $w(\cdot),v(\cdot)$ taking values in $\ell_{\infty}(I)^+$, has the MLIM property.%
\end{enumerate}
Then $\Sigma$ is ISS.
\end{theorem}

\begin{proof}
We show that $\Sigma$ is UGS and satisfies the bUAG property, which implies ISS by Lemma \ref{lem_UGS_and_bUAG_imply_UAG}.%

{\bf UGS}. This follows from Theorem \ref{thm_ugs_semimax_sg}. Indeed, the assumptions (i) and (ii) of Theorem \ref{thm_ugs_semimax_sg} are satisfied with $\sigma_i(r) := \beta_i(r,0) \in \K$ and the gains $\gamma_{ij},\gamma_i$ from the ISS estimates for $\Sigma_i$, $i\in I$. From Proposition \ref{prop_mlim_implies_mbip} and Assumption (iii) of this theorem, it follows that Assumption (iii) of Theorem \ref{thm_ugs_semimax_sg} is satisfied. Hence, $\Sigma$ is forward complete and UGS.%

{\bf bUAG}. As $\Sigma$ is the interconnection of the systems $\Sigma_i$ and since $\Sigma$ is forward complete, we have $\phi_i(t,x,u) = \bar{\phi}_i(t,x_i,(\phi_{\neq i},u))$ for all $(t,x,u) \in \R_+ \tm X \tm \Uc$ and $i \in I$, with the notation from Definition \ref{def_interconnection}.%

Pick any $r > 0$, any $u \in \overline{B}_{r,\Uc}$ and $x \in \overline{B}_{r,X}$. As $\Sigma$ is UGS, there are $\sigma^{\UGS},\gamma^{\UGS} \in \Kinf$ so that%
\begin{equation*}
  \|\phi(t,x,u)\|_X \leq \sigma^{\UGS}(r) + \gamma^{\UGS}(r) =: \mu(r) \mbox{\quad for all\ } t \geq 0.%
\end{equation*}
In view of the cocycle property, for all $i \in I$ and $t,\tau \geq 0$ we have%
\begin{align*}
  \phi_i(t + \tau,x,u) &= \bar{\phi}_i(t+\tau,x_i,(\phi_{\neq i},u)) \\
	                     &= \bar{\phi}_i(\tau,\bar{\phi}_i(t,x_i,(\phi_{\neq i},u)),(\phi_{\neq i}(\cdot+t),u(\cdot+t))).%
\end{align*}
Given $\ep>0$, choose $\tau^* = \tau^*(\ep,r) \geq 0$ such that $\beta_{\max}(\mu(r),\tau^*) \leq \ep$. Then%
\begin{align}\label{eq_firstest}
\begin{split}
  x \in \overline{B}_{r,X} &\wedge u \in \overline{B}_{r,\Uc} \wedge \tau \geq \tau^* \wedge t \geq 0 \\
	\Rightarrow &\|\phi_i(t+\tau,x,u)\|_{X_i} \leq \beta_i(\|\bar{\phi}_i(t,x_i,(\phi_{\neq i},u))\|_{X_i},\tau) \\
	                                 &\qquad + \sup_{j \in I} \gamma_{ij}( \|\phi_j\|_{[t,t+\tau]} ) + \gamma_i(\|u(\cdot+t)\|_{\Uc}) \\
																	&\leq \beta_{\max}(\|\phi(t,x,u)\|_X,\tau^*) + \sup_{j \in I}\gamma_{ij}(\|\phi_j\|_{[t,\infty)}) + \gamma_i(\|u\|_{\Uc}) \\
																	&\leq \ep + \sup_{j \in I}\gamma_{ij}(\|\phi_j\|_{[t,\infty)}) + \gamma_i(\|u\|_{\Uc}).%
\end{split}
\end{align}
Now pick any $k \in \N$ and write%
\begin{equation*}
  B(r,k) := \overline{B}_{r,X} \tm \{ u \in \Uc : \|u\|_{\Uc} \in [2^{-k}r,2^{-k+1}r]\}.%
\end{equation*} 
Then, taking the supremum in the above inequality over all $(x,u) \in B(r,k)$, we obtain for all $i \in I$ and all $t \geq 0$ that%
\begin{equation*}
  \sup_{(x,u) \in B(r,k)}\|\phi_i(t+\tau^*,x,u)\|_{X_i} \leq \ep + \sup_{j \in I} \gamma_{ij}\Bigl( \sup_{(x,u) \in B(r,k)} \|\phi_j\|_{[t,\infty)} \Bigr) + \gamma_i(2^{-k+1}r).%
\end{equation*}
This implies for all $t \geq 0$ that%
\begin{align*}
  &\sup_{s \geq t + \tau^*}\sup_{(x,u) \in B(r,k)} \|\phi_i(s,x,u)\|_{X_i} \\
	&\quad \leq \ep + \sup_{j \in I} \gamma_{ij}\Bigl(\sup_{s \geq t} \sup_{(x,u) \in B(r,k)} \|\phi_j(s,x,u)\|_{X_j}\Bigr) + \gamma_i(2^{-k+1}r).%
\end{align*}
Now we define%
\begin{equation*}
  w_i(t,r,k) := \sup_{s \geq t} \sup_{(x,u) \in B(r,k)} \|\phi_i(s,x,u)\|_{X_i}%
\end{equation*}
and note that $w_i(t,r,k) \in [0,\mu(r)]$ for all $i \in I$ and $t \geq 0$. With this notation, we can rewrite the preceding inequality as%
\begin{equation*}
  w_i(t + \tau^*,r,k) \leq \ep + \sup_{j \in I} \gamma_{ij}(w_i(t,r,k)) + \gamma_i(2^{-k+1}r).%
\end{equation*}
Using vector notation $\vec{w}(t,r,k) := (w_i(t,r,k))_{i\in I}$ and $\vec{\gamma}(r) := (\gamma_i(r))_{i \in I}$, this can be written as%
\begin{equation*}
  \vec{w}(t+\tau^*,r,k) \leq \Gamma_{\otimes}(\vec{w}(t,r,k)) + \ep {\bf 1} + \vec{\gamma}(2^{-k+1}r).%
\end{equation*}
Observe that $\vec{w}(t,r,k) \in \ell_{\infty}(I)^+$, as the entries of the vector are uniformly bounded by $\mu(r)$, and $\vec{w}(t_2,r,k) \leq \vec{w}(t_1,r,k)$ for $t_2 \geq t_1$. Hence, $\vec{w}(l) := \vec{w}(l\tau^*,r,k)$, $l \in \Z_+$, is a monotone solution of \eqref{eq_sg_system} for the constant input $v(\cdot) \equiv \ep {\bf 1} + \vec{\gamma}(2^{-k+1}r)$. By assumption (iii) of the theorem, this implies the existence of a time $\tilde{\tau} = \tilde{\tau}(\ep,r,k)$ and a $\Kinf$-function $\xi$ such that%
\begin{align*}
  \|\vec{w}(\tilde{\tau},r,k)\|_{\ell_{\infty}(I)} &\leq \ep + \xi(\|\ep {\bf 1} + \vec{\gamma}(2^{-k+1}r)\|_{\ell_{\infty}(I)}) \\
	                                              &\leq \ep + \xi(\|\ep{\bf 1}\|_{\ell_{\infty}(I)} + \|\vec{\gamma}(2^{-k+1}r)\|_{\ell_{\infty}(I)}) \\
																								&\leq \ep + \xi(\ep + \gamma_{\max}(2^{-k+1}r)) \\
																								&\leq \ep + \xi(2\ep) + \xi(2\gamma_{\max}(2^{-k+1}r)).%
\end{align*}
By definition, this implies%
\begin{align*}
  &i \in I \wedge (x,u) \in B(r,k) \wedge t \geq \tilde{\tau}(\ep,r,k) \\
	&\Rightarrow \|\phi_i(t,x,u)\|_{X_i} \leq \ep + \xi(2\ep) + \xi(2\gamma_{\max}(2^{-k+1}r)).%
\end{align*}
Now define $k_0 = k_0(\ep,r)$ as the minimal $k \geq 1$ so that $\xi(2\gamma_{\max}(2^{1-k}r)) \leq \ep$ and let%
\begin{equation*}
  \hat{\tau}(\ep,r) := \max\{ \tilde{\tau}(\ep,r,k) : 1 \leq k \leq k_0(\ep,r) \}.%
\end{equation*}
Pick any $0 \neq u \in \overline{B}_{r,\Uc}$. Then there is $k \in \N$ with $\|u\|_{\Uc} \in (2^{-k}r,2^{-k+1}r]$. If $k \leq k_0$ (large input), then for $t \geq \hat{\tau}(\ep,r)$ we have%
\begin{align}\label{eq_349}
\begin{split}
  \|\phi(t,x,u)\|_X &\leq \ep + \xi(2\ep) + \xi(\gamma_{\max}(2^{-k+1}r)) \\
	&\leq \ep + \xi(2\ep) + \xi(2\gamma_{\max}(2\|u\|_{\Uc})).%
\end{split}
\end{align}
It remains to consider the case when $k > k_0$ (small input). For any $q \in [0,r]$, one can take the supremum in \eqref{eq_firstest} over $x \in \overline{B}_{r,X}$ and $u \in \overline{B}_{q,\Uc}$ to obtain%
\begin{align*}
  &\sup_{(x,u) \in \overline{B}_{r,X} \tm \overline{B}_{q,\Uc}}\|\phi_i(t+\tau,x,u)\|_{X_i} \\
	&\qquad \leq \ep + \sup_{j \in I}\gamma_{ij}\Bigl(\sup_{(x,u) \in \overline{B}_{r,X} \tm \overline{B}_{q,\Uc}}\|\phi_j\|_{[t,\infty)}\Bigr) + \gamma_i(q).%
\end{align*}
With $z_i(t,r,q) := \sup_{s \geq t}\sup_{(x,u) \in \overline{B}_{r,X} \tm \overline{B}_{q,\Uc}}\|\phi_i(s,x,u)\|_{X_i}$, analogous steps as above lead to the following: for every $\ep>0$, $r>0$ and $q \in [0,r]$ there is a time $\bar{\tau} = \bar{\tau}(\ep,r,q)$ such that%
\begin{align*}
  (x,u) \in \overline{B}_{r,X} \tm \overline{B}_{q,\Uc} \wedge t \geq \bar{\tau} \quad \Rightarrow \quad \|\phi(t,x,u)\|_X \leq \ep + \xi(2\ep) + \xi(2\gamma_{\max}(q)).%
\end{align*}
In particular, for $q_0 := 2^{-k_0(\ep,r)+1}$, we have%
\begin{equation}\label{eq_352}
  (x,u) \in \overline{B}_{r,X} \tm \overline{B}_{q_0,\Uc} \wedge t \geq \bar{\tau} \quad \Rightarrow \quad \|\phi(t,x,u)\|_X \leq 2\ep + \xi(2\ep),%
\end{equation}
since $\xi(2\gamma_{\max}(q_0)) = \xi(2\gamma_{\max}(2^{-k_0(\ep,r)+1})) \leq \ep$ by definition of $k_0$. Define $\tau(\ep,r) := \max\{\hat{\tau}(\ep,r),\bar{\tau}(\ep,r,q_0)\}$. Combining \eqref{eq_349} and \eqref{eq_352}, we obtain%
\begin{align*}
  &(x,u) \in \overline{B}_{r,X} \tm \overline{B}_{r,\Uc} \wedge t \geq \tau(\ep,r) \\
	&\qquad \Rightarrow \quad \|\phi(t,x,u)\|_X \leq 2\ep + \xi(2\ep) + \xi(2\gamma_{\max}(2\|u\|_{\Uc})).%
\end{align*}
As $r \mapsto \xi(2\gamma_{\max}(2r))$ is a $\Kinf$-function, we have proved that $\Sigma$ has the bUAG property which completes the proof. \qed%
\end{proof}

For finite networks, Theorem~\ref{thm_smallgain_iss_semimax} was shown in \cite{Mir19b}. However, in the proof of the infinite-dimensional case there are essential novelties, which are due to the fact that the trajectories of an infinite number of subsystems do not necessarily have a uniform speed of convergence. This resulted also in a strengthening of the employed small-gain condition.

In the special case when all interconnection gains $\gamma_{ij}$ are linear, the small-gain condition in our theorem can be formulated more directly in terms of the gains, as the following corollary shows.%

\begin{corollary}
\label{cor:SGT-sup-linear-gains}{\textbf{(Linear ISS small-gain theorem in semi-maximum formulation)}}
Given an interconnection $(\Sigma,\Uc,\phi)$ of systems $\Sigma_i$ as in Theorem \ref{thm_smallgain_iss_semimax}, additionally to the assumptions (i) and (ii) of this theorem, assume that all gains $\gamma_{ij}$ are linear functions (and hence can be identified with nonnegative real numbers), $\Gamma_{\otimes}$ is well-defined and the following condition holds:%
\begin{equation}\label{eq_sg_hsr_cond}
  \lim_{n \rightarrow \infty} \Bigl( \sup_{j_1,\ldots,j_{n+1}\in I} \gamma_{j_1j_2} \cdots \gamma_{j_{n}j_{n+1}}\Bigr)^{1/n} < 1.%
\end{equation}
Then $\Sigma$ is ISS.
\end{corollary}

\begin{proof}
We only need to show that Assumption (iii) of Theorem \ref{thm_smallgain_iss_semimax} is implied by \eqref{eq_sg_hsr_cond}. The linearity of the gains $\gamma_{ij}$ implies that the operator $\Gamma_{\otimes}$ is homogeneous of degree one and subadditive, see Remark \ref{rem_lg_sg_operator}. Then Proposition \ref{prop:eISS-criterion-homogeneous-systems} and Remark \ref{rem_homogen_spectral_radius} together show that  \eqref{eq_sg_hsr_cond} implies that the system%
\begin{equation*}
  w(k+1) \leq \Gamma_{\otimes}(w(k)) + v(k)%
\end{equation*}
is eISS (according to Definition~\ref{def:eISS-discrete-time-inequalities}), which easily implies the MLIM property for this system. \qed%
\end{proof}

\subsection{Small-gain theorems in summation formulation}\label{sec:Small-gain theorems in summation formulation}

Now we formulate the small-gain theorems for UGS and ISS in summation formulation.%

\begin{theorem}\label{thm_ugs_summation_sg}{\textbf{(UGS small-gain theorem in summation formulation)}}
Let $I$ be a countable index set, $(X_i,\|\cdot\|_{X_i})$, $i\in I$, be normed spaces and $\Sigma_i = (X_i,\PC_b(\R_+,X_{\neq i}) \tm \Uc,\bar{\phi}_i)$, $i\in I$ be forward complete control systems. Assume that the interconnection $\Sigma = (X,\Uc,\phi)$ of the systems $\Sigma_i$ is well-defined. Furthermore, let the following assumptions be satisfied:%
\begin{enumerate}
\item[(i)] Each system $\Sigma_i$ is UGS in the sense of Definition \ref{def_subsys_ugs_sum} (summation formulation) with $\sigma_i \in \K$ and nonlinear gains $\gamma_{ij},\gamma_i \in \K \cup \{0\}$.%
\item[(ii)] There exist $\sigma_{\max} \in \Kinf$ and $\gamma_{\max} \in \Kinf$ so that $\sigma_i \leq \sigma_{\max}$ and $\gamma_i \leq \gamma_{\max}$, pointwise for all $i \in I$.%
\item[(iii)]  Assumption \ref{ass_gammasum_welldef}  is satisfied for the operator $\Gamma_{\boxplus}$ defined via the gains $\gamma_{ij}$ from (i) and $\id - \Gamma_{\boxplus}$ has the MBI property.%
\end{enumerate}
Then $\Sigma$ is forward complete and UGS.%
\end{theorem}

\begin{proof}
The proof is exactly the same as for Theorem \ref{thm_ugs_semimax_sg}, with the operator $\Gamma_{\boxplus}$ in place of $\Gamma_{\otimes}$. \qed%
\end{proof}

\begin{theorem}\label{thm_smallgain_iss_summation}{\textbf{(Nonlinear ISS small-gain theorem in summation formulation)}}
Let $I$ be a countable index set, $(X_i,\|\cdot\|_{X_i})$, $i\in I$ be normed spaces and $\Sigma_i = (X_i,\PC_b(\R_+,X_{\neq i}) \tm \Uc,\bar{\phi}_i)$, $i\in I$ be forward complete control systems. Assume that the interconnection $\Sigma = (X,\Uc,\phi)$ of the systems $\Sigma_i$ is well-defined. Furthermore, let the following assumptions be satisfied:%
\begin{enumerate}
\item[(i)] Each system $\Sigma_i$ is ISS in the sense of Definition \ref{def_subsys_iss_sum} with $\beta_i \in \KL$ and nonlinear gains $\gamma_{ij},\gamma_i \in \K \cup \{0\}$.%
\item[(ii)] There are $\beta_{\max} \in \KL$ and $\gamma_{\max} \in \K$ so that $\beta_i \leq \beta_{\max}$ and $\gamma_i \leq \gamma_{\max}$, pointwise for all $i \in I$.%
\item[(iii)] Assumption \eqref{ass_gammasum_welldef} holds and the discrete-time system%
\begin{equation}
\label{eq:Gamma-boxplus-discrete-time}
  w(k+1) \leq \Gamma_{\boxplus}(w(k)) + v(k),%
\end{equation}
with $w(\cdot),v(\cdot)$ taking values in $\ell_{\infty}(I)^+$, has the MLIM property.%
\end{enumerate}
Then $\Sigma$ is ISS.
\end{theorem}

{
\begin{proof}
The proof is almost completely the same as for Theorem \ref{thm_smallgain_iss_semimax}. The only difference is that instead of interchanging the order of two suprema $\sup_{s \geq t}$ and $\sup_{j \in I}$, we now have to use the estimate
$\sup_{s \geq t} \sum_{j \in I} \ldots \leq \sum_{j \in I} \sup_{s \geq t} \ldots$,
which is trivially satisfied. \qed%
\end{proof}
}

Again, we formulate a corollary for the case when all gains $\gamma_{ij}$ are linear.%

\begin{corollary}\label{cor:SGT-sum-linear-gains}{\textbf{(Linear ISS small-gain theorem in summation formulation)}}
Given an interconnection $(\Sigma,\Uc,\phi)$ of systems $\Sigma_i$ as in Theorem \ref{thm_smallgain_iss_summation}, additionally to the assumptions (i) and (ii) of this theorem, assume that all gains $\gamma_{ij}$ are linear functions (and hence can be identified with nonnegative real numbers), the linear operator $\Gamma_{\boxplus}$ is well-defined (thus bounded) and satisfies the spectral radius condition $r(\Gamma_{\boxplus}) < 1$. Then $\Sigma$ is ISS.
\end{corollary}

\begin{proof}
By Proposition \ref{prop:eISS-criterion-linear-systems}, $r(\Gamma_{\boxplus}) < 1$ is equivalent to the MLIM property of the system
\eqref{eq:Gamma-boxplus-discrete-time}, hence to Assumption (iii) of Theorem \ref{thm_smallgain_iss_summation}. \qed%
\end{proof}

\subsection{Example: a linear spatially invariant system}\label{examp: a linear spatially invariant system}

Let us analyze the stability of a spatially invariant infinite network
\begin{equation}\label{eq:linear-spatially-invariant-system}
  \dot{x}_i = ax_{i-1} - x_i + b x_{i+1} + u,\quad i\in\Z,%
\end{equation}
where $a,b>0$ and each $\Sigma_i$ is a scalar system with the state $x_i \in\R$, internal inputs $x_{i-1}$, $x_{i+1}$ and an external input $u$, belonging to the input space $\Uc:=L_\infty(\R_+,\R)$.

Following the general approach in Section \ref{sec:Infinite interconnections}, we define the state space for the interconnection of 
$(\Sigma_i)_{i\in\Z}$ as $X:=\ell_\infty(\Z)$.
Similarly as for finite-dimensional ODEs, it is possible to introduce the concept of a mild (Carath\'eodory) solution for the equation 
\eqref{eq:linear-spatially-invariant-system}, for which we refer, e.g., to \cite{KMS19}. As \eqref{eq:linear-spatially-invariant-system} is linear, it is easy to see that for each initial condition $x_0 \in X$ and for each input $u \in\Uc$ the corresponding mild solution is unique and exists on $\R_+$. We denote it by $\phi(\cdot,x_0,u)$. One can easily check that the triple $\Sigma:=(X,\Uc,\phi)$ defines a well-posed and forward complete interconnection in the sense of this paper.%

Having a well-posed control system $\Sigma$, we proceed to its stability analysis.%

\begin{proposition}\label{prop:Stability-linear-systems} 
The coupled system \eqref{eq:linear-spatially-invariant-system} is ISS if and only if $a+b<1$.
\end{proposition}

\begin{proof}
\q{$\Rightarrow$}: For any $a,b>0$, the function $y: t \mapsto (\mathrm{e}^{(a+b-1)t}x^*)_{i\in\Z}$ is a solution of \eqref{eq:linear-spatially-invariant-system} subject to an initial condition $(x^*)_{i\in\Z}$ and input $u\equiv 0$. This shows that $a+b \geq 1$ implies that the system \eqref{eq:linear-spatially-invariant-system} is not ISS.%

\q{$\Leftarrow$}: By variation of constants, we see that for any $i\in\Z$, treating $x_{i-1}, x_{i+1}$ as external inputs from $L_\infty(\R_+,\R)$, we have the following ISS estimate for the $x_i$-subsystem:%
\begin{align*}
  |x_i(t)| &= \Big|\mathrm{e}^{-t}x_i(0) + \int_0^t \mathrm{e}^{s-t}[a x_{i-1}(s) + b x_{i+1}(s) + u(s)] \rmd s\Big|\\
  			   &\leq \mathrm{e}^{-t}|x_i(0)| + a \|x_{i-1}\|_\infty  + b \|x_{i+1}\|_\infty + \|u\|_\infty,
\end{align*}
for any $t\geq 0$, $x_i(0)\in\R$ and all $x_{i-1}, x_{i+1},u \in L_\infty(\R_+,\R)$.%

This shows that the $x_i$-subsystem is ISS in summation formulation and the corresponding gain operator is a linear operator $\Gamma:\ell_\infty^+(\Z) \to \ell_\infty^+(\Z)$, acting on $s=(s_i)_{i\in\Z}$ as $\Gamma(s) = (as_{i-1} + b s_{i+1})_{i\in\Z}$. It is easy to see that
{
\begin{align*}
\|\Gamma\| 
&:= \sup_{\|s\|_{\ell_\infty(\Z)}=1}\|\Gamma s\|_{\ell_\infty(\Z)}
=  \|\Gamma {\bf 1}\|_{\ell_\infty(\Z)}
= a+b <1,%
\end{align*}
}
and thus $r(\Gamma)<1$, and the network is ISS by Corollary~\ref{cor:SGT-sum-linear-gains}. \qed%
\end{proof}

\subsection{Example: a nonlinear spatially invariant system}\label{examp: a nonlinear spatially invariant system}

Consider an infinite interconnection (in the sense of the previous sections)%
\begin{equation}\label{eq:cubic-spatially-invariant-system}
  \dot{x}_i = - x_i^3 + \max\{ax_{i-1}^3,b x_{i+1}^3,u \},\quad i\in\Z,%
\end{equation}
where $a,b>0$. As in Section~\ref{examp: a linear spatially invariant system}, each $\Sigma_i$ is a scalar system with the state $x_i \in\R$, internal inputs $x_{i-1}$, $x_{i+1}$ and an external input $u$, belonging to the input space $\Uc:=L_\infty(\R_+,\R)$. Let the state space for the interconnection $\Sigma$ be $X:=\ell_\infty(\Z)$.%

First we analyze the well-posedness of the interconnection \eqref{eq:cubic-spatially-invariant-system}. Define for $x = (x_i)_{i\in\Z} \in X$ and $v \in \R$%
\begin{equation*}
  f_i(x,v):= - x_i^3 + \max\{ax_{i-1}^3,b x_{i+1}^3,v \},\quad i \in\Z,%
\end{equation*}
as well as%
\begin{equation*}
  f(x,v):=(f_i(x,v))_{i\in\Z} \in \R^{\Z}.%
\end{equation*}
It holds that%
\begin{equation*}
  |f_i(x,v)| \leq  \|x\|_X^3 + {\max\{a,b\}}\max\{\|x\|^3_X,|v|\},%
\end{equation*}
and thus $f(x,v) \in X$ with $\|f(x,v)\|_X \leq  \|x\|_X^3 + {\max\{a,b\}} \max\{\|x\|^3_X,|v|\}$.%

Furthermore, $f$ is clearly continuous in the second argument. Let us show Lipschitz continuity of $f$ on bounded balls with respect to the first argument. For any $x = (x_i)_{i\in\Z} \in X$, $y = (y_i)_{i\in\Z} \in X$ and any $v \in \R$ we have%
\begin{align*}
\|f(x,u)-&f(y,u)\|_X 
= \sup_{i\in\Z}|f_i(x,u)-f_i(y,u)|\\
&= \sup_{i\in\Z}\big|- x_i^3 + \max\{ax_{i-1}^3,b x_{i+1}^3,v \}  + y_i^3 - \max\{ay_{i-1}^3,b y_{i+1}^3,v \}\big|\\
&\leq \sup_{i\in\Z}\big| x_i^3 - y_i^3\big|
+ \sup_{i\in\Z}\big|\max\{ax_{i-1}^3,b x_{i+1}^3,v \}  - \max\{ay_{i-1}^3,b y_{i+1}^3,v \}\big|.%
\end{align*}
By Birkhoff's inequality $|\max\{a_1,a_2,a_3\} - \max\{b_1,b_2,b_3\}|\leq \sum_{i=1}^3|a_i-b_i|$, which holds for all real $a_i,b_i$,  we obtain
%
\begin{align*}
\|f(x,u)&-f(y,u)\|_X 
\leq \sup_{i\in\Z}\big| x_i^3 - y_i^3\big|
+ a\sup_{i\in\Z}\big|x_{i-1}^3 -y_{i-1}^3\big|
+ b\sup_{i\in\Z}\big|x_{i+1}^3 -y_{i+1}^3\big|\\
&= (1+a+b)\sup_{i\in\Z}\big| x_i^3 - y_i^3\big|
\leq (1+a+b) \sup_{i\in\Z}\big| x_i - y_i\big| \sup_{i\in\Z}\big| x_i^2 +x_iy_i + y_i^2\big| \\
&\leq (1+a+b)\|x-y\|_X \big(\|x\|^2_X + \|x\|_X\|y\|_X + \|y\|^2_X\big),
\end{align*}
%
which shows Lipschitz continuity of $f$ with respect to the first argument on the bounded balls in $X$, uniformly with respect to the second argument.%

According to \cite[Thm.~2.4]{AuW96},\footnote{The cited result assumes a global Lipschitz condition and accordingly ensures forward completeness. However, via the retraction method this result can easily be localized.} this ensures that the Carath\'eodory solutions of \eqref{eq:cubic-spatially-invariant-system} exist locally, are unique for any fixed initial condition $x_0\in X$ and external input $u\in\Uc$. We denote the corresponding maximal solution by $\phi(\cdot,x_0,u)$. One can easily check that the triple $\Sigma:=(X,\Uc,\phi)$ defines a well-posed interconnection in the sense of this paper, and furthermore $\Sigma$ has BIC property (cf.~\cite[Thm.~4.3.4]{CaH98}).%

We proceed to the stability analysis:
\begin{proposition}
\label{prop:Stability-cubic-systems} 
The coupled system \eqref{eq:cubic-spatially-invariant-system} is ISS if and only if $\max\{a,b\}<1$.
\end{proposition}

\begin{proof}
\q{$\Rightarrow$}: For any $a,b>0$ consider the scalar equation %
\begin{equation*}
  \dot{z} = - (1-\max\{a,b\})z^3,%
\end{equation*}
subject to an initial condition $z(0)=x^*$. The function $t \mapsto (z(t))_{i\in\Z}$ is a solution of \eqref{eq:cubic-spatially-invariant-system} subject to an initial condition $(x^*)_{i\in\Z}$ and input $u\equiv 0$. This shows that for $\max\{a,b\} \geq 1$ the system \eqref{eq:cubic-spatially-invariant-system} is not ISS.

\q{$\Leftarrow$}:
Consider $x_{i-1}$, $x_{i+1}$ and $u$ as inputs to the $x_i$-subsystem of \eqref{eq:cubic-spatially-invariant-system}
and define $q:=\max\{ax_{i-1}^3,b x_{i+1}^3,u \}$.
The derivative of $|x_i(\cdot)|$ along the trajectory satisfies for almost all $t$ the following inequality:
\begin{equation*}
  \frac{\mathrm{d}}{\mathrm{d}t}|x_i(t)|\leq -|x_i(t)|^3 + q(t)   \leq -|x_i(t)|^3 + \|q\|_\infty.%
\end{equation*}
For any $\varepsilon>0$, if $\|q\|_\infty \leq \frac{1}{1+\varepsilon}|x_i(t)|^3$, we obtain%
\begin{equation*}
  \frac{\mathrm{d}}{\mathrm{d}t}|x_i(t)|\leq -\frac{\varepsilon}{1+\varepsilon}|x_i(t)|^3.
\end{equation*}
Arguing as in the proof of direct Lyapunov theorems ($x_i \mapsto |x_i|$ is an ISS Lyapunov function for the $x_i$-subsystem), see, e.g., \cite[Lem.~2.14]{SoW95}, we obtain that there is a certain $\beta\in\KL$
such that for all $t\geq 0$ it holds that
\begin{align*}
|x_i(t)| & \leq \beta(|x_i(0)|,t) +  \big((1+\varepsilon)\|q\|_\infty \big)^{1/3}\\
& = \beta(|x_i(0)|,t) +  \max\{a_1 \|x_{i-1}\|_\infty,b_1 \|x_{i+1}\|_\infty,(1+\varepsilon)^{1/3}\|u\|_\infty^{1/3} \}\\
& \leq \beta(|x_i(0)|,t) +  \max\{a_1 \|x_{i-1}\|_\infty,b_1 \|x_{i+1}\|_\infty\} +(1+\varepsilon)^{1/3}\|u\|_\infty^{1/3},
\end{align*}
where $a_1=(1+\varepsilon)^{1/3}a^{1/3}$, $b_1=(1+\varepsilon)^{1/3}b^{1/3}$.%

This shows that the $x_i$-subsystem is ISS in semi-maximum formulation with the corresponding homogeneous {of degree one} gain operator 
$\Gamma:\ell_\infty^+(\Z) \to \ell_\infty^+(\Z)$ given for all $s=(s_i)_{i\in\Z}$ by $\Gamma(s) = (\max\{a_1 s_{i-1}, b_1 s_{i+1}\})_{i\in\Z}$.%

The previous computations are valid for all $\varepsilon>0$. Now pick $\varepsilon>0$ such that $a_1<1$ and $b_1<1$, which is possible
as $a \in (0,1)$ and $b\in(0,1)$. The ISS of the network follows by Corollary~\ref{cor:SGT-sup-linear-gains}. \qed%
\end{proof}

\section{Small-gain conditions}\label{sec:Small-gain conditions}

Key assumptions in the ISS and UGS small-gain theorems are the monotone limit property and monotone bounded invertibility property, respectively. In this section, we thoroughly investigate these properties. More precisely, in Section~\ref{sec:A uniform small-gain condition and the MBI property} we characterize the MBI property in terms of the uniform small-gain condition, in Section~\ref{sec:Non-uniform small-gain conditions}, we relate the uniform small-gain condition to several types of non-uniform small-gain conditions which have already been exploited in the small-gain analysis of finite and infinite networks. In Section~\ref{sec:Finite-dimensional systems}, we derive new relationships between small-gain conditions in the finite-dimensional case.
{Finally, in Section~\ref{sec:Systems with linear gains}, we provide efficient criteria for the MLIM and the MBI property in case of linear operators and operators of the form $\Gamma_\otimes$ induced by linear gains.}

\subsection{A uniform small-gain condition and the MBI property}\label{sec:A uniform small-gain condition and the MBI property}

As we have seen in Section~\ref{sec:Small-gain theorems}, the monotone bounded invertibility is a crucial property for the small-gain analysis of finite and infinite networks. The next proposition yields small-gain type criteria for the MBI property. Although in the context of small-gain theorems in terms of trajectories, derived in this paper, we are interested primarily in the case of $(X,X^+) = (\ell_\infty(I),\ell^+_\infty(I))$, we prove the results in a more general setting, which besides the mathematical appeal  also has important applications to Lyapunov-based small-gain theorems for infinite networks, where other choices for $X$ are useful, see, e.g., \cite{KMS19} where $X=\ell_p$ for finite $p\geq1$.%

\begin{proposition}\label{prop:criteria-MBI-without-unit}
Let $(X,X^+)$ be an ordered Banach space with a generating cone $X^+$. For every nonlinear operator $A: X^+ \to X^+$, the following conditions are equivalent:
\begin{enumerate}[label=(\roman*)]
\item\label{itm:MBI-criterion-without-unit-1} $\id - A$ satisfies the MBI property.
\item\label{itm:MBI-criterion-without-unit-2} The \emph{uniform small-gain condition} holds: There exists $\eta \in \Kinf$ such that%
\begin{equation}
\label{eq:uSGC-dist-form}
  \dist(A(x) - x,X^+) \geq \eta(\|x\|_X) \mbox{\quad for all\ } x \in X^+.%
\end{equation}
\end{enumerate}
\end{proposition}

\begin{proof}
\ref{itm:MBI-criterion-without-unit-1} $\Rightarrow$ \ref{itm:MBI-criterion-without-unit-2}. Fix $x \in X^+$ and write $a := (A - \id)(x)$. Let $\varepsilon > 0$. We choose $z \in X^+$ such that $\|a-z\|_X \le \dist(a,X^+) + \varepsilon$ and we set $y := a-z$. If the constant $M > 0$ is chosen as in~\eqref{eq_bounded_decomposition}, we can decompose $y$ as $y = u-v$, where $u,v \in X^+$ and $\|u\|_X,\|v\|_X \le M \|y\|_X \le M\dist(a,X^+) + M\varepsilon$. Then we have
\begin{equation*}
  (\id - A)(x) = -a = -y-z = v - (u+z) \le v,
\end{equation*}
so it follows from 
{the MBI property of $\id - A$} that
\begin{equation*}
  \|x\|_X \le \xi(\|v\|_X) \le \xi\big(M \dist(a,X^+)+M\varepsilon\big).
\end{equation*}
Consequently,
\begin{equation*}
  \dist(a,X^+) \ge \frac{1}{M} \xi^{-1}(\|x\|_X) - \varepsilon.
\end{equation*}
Since $\varepsilon$ was arbitrary, this implies~(ii) with $\eta := \frac{1}{M}\xi^{-1}$.

\ref{itm:MBI-criterion-without-unit-2} $\Rightarrow$ \ref{itm:MBI-criterion-without-unit-1}. Let $v,w \in X^+$ and $(\id-A)(v) \le w$. The vector $z := w + (A-\id)(v)$ is positive, so from (ii) it follows that%
\begin{equation*}
  \eta(\|v\|_X) \leq \dist\big( (A-\id)(v), X^+ \big) \le \|(A-\id)(v) - z\|_X = \|-w\|_X = \|w\|_X.%
\end{equation*}
Hence, $\|v\|_X \le \eta^{-1}(\|w\|_X)$. \qed%
\end{proof}

\begin{remark}
The uniform small-gain condition in {Proposition~\ref{prop:criteria-MBI-without-unit}\ref{itm:MBI-criterion-without-unit-2}} is a uniform version of the well-known small-gain condition, sometimes also called \emph{no-joint-increase condition}:%
\begin{equation*}
  A(x) \not\geq x \mbox{\quad for all\ } x \in X^+ \setminus \{0\}.%
\end{equation*}
Indeed, $A(x) \not\geq x$ is equivalent to $A(x) - x \not\geq 0$, which in turn is equivalent to $\dist(A(x) - x, X^+) > 0$.
\end{remark}

\begin{remark}
It is important to point out that the distance to the positive cone which occurs in the uniform small-gain condition in Proposition~\ref{prop:criteria-MBI-without-unit} can be explicitly computed on many concrete spaces. Indeed, many important real-valued sequence or function spaces such as $X = \ell_p$ or $X = L_p(\Omega,\mu)$ (for $p \in [1,\infty]$ and a measure space $(\Omega,\mu)$) are not only ordered Banach spaces but so-called \emph{Banach lattices}.%

An ordered Banach space $(X,X^+)$ is called a \emph{Banach lattice} if, for all $x \in X$, the set $\{-x,x\}$ has a smallest upper bound in $X$, which is usually called the \emph{modulus} of $x$ and denoted by $|x|$, and if $\|x\|_X \le \|y\|_X$ whenever $|x| \le |y|$. In concrete sequence and function spaces, the modulus of a function is just the pointwise (respectively, almost everywhere) modulus.%

Now, assume that $(X,X^+)$ is a Banach lattice and let $x \in X$. Then the vectors $x^+ := \frac{|x|+x}{2} \ge 0$ and $x^- := \frac{|x|-x}{2} \ge 0$ are called the \emph{positive} and \emph{negative part} of $x$, respectively; clearly, they satisfy $x^+ - x^- = x$ and $x^+ + x^- = |x|$. If $X$ is a concrete sequence or function space, then $x^-$ is simply $0$ at all points where $x$ is positive, and equal to $-x$ at all points where $x$ is negative.%

In a Banach lattice $(X,X^+)$, we have the formula%
\begin{equation*}
  \dist(x,X^+) = \|x^-\|_X%
\end{equation*}
for each $x \in X$, as can easily be verified.
\end{remark}




If the cone of the ordered Banach space $(X,X^+)$ has nonempty interior, the uniform small-gain condition from Proposition~\ref{prop:criteria-MBI-without-unit} can also be expressed by a condition that involves a fixed interior point of $X^+$.%

\begin{proposition}\label{prop:criteria-MBI-with-unit}
Let $(X,X^+)$ be an ordered Banach space, assume that the cone $X^+$ has nonempty interior and let $z$ be an interior point of $X^+$. For every nonlinear operator $A: X^+ \to X^+$, the following conditions are equivalent:%
\begin{enumerate}[label=(\roman*)]
\item\label{itm:MBI-criterion-with-unit-1} There is $\eta \in \Kinf$ such that%
\begin{equation}
\label{eq:uSGC-with-unit}
  A(x) \not\geq x - \eta(\|x\|_X)z \mbox{\quad for all\ } x \in X^+ \setminus \{0\}.%
\end{equation}
\item\label{itm:MBI-criterion-with-unit-2} The uniform small-gain condition from Proposition~\ref{prop:criteria-MBI-without-unit}\ref{itm:MBI-criterion-without-unit-2} holds.
\end{enumerate}
\end{proposition}

\begin{proof}
\ref{itm:MBI-criterion-with-unit-1} $\Rightarrow$ \ref{itm:MBI-criterion-with-unit-2}. Let \ref{itm:MBI-criterion-with-unit-1} hold with some $\eta\in\Kinf$. By \cite[Prop.~2.11]{GW20}, we can find a number $c>0$ such that for every $y \in X$ we have%
\begin{equation}\label{eq_oldlem3}
  \|y\|_X \le c \quad \Rightarrow \quad y \geq -z.%
\end{equation}
Assume towards a contradiction that~(ii) does not hold. Then \eqref{eq:uSGC-dist-form} fails, in particular, for the function $c\eta$. Thus, we can infer that there is $x \in X^+ \setminus \{0\}$ so that 
\begin{equation*}
  \dist\big( (A-\id)(x), X^+ \big) < c\eta(\|x\|_X).
\end{equation*}
Hence, there exists $y \in X^+$ such that
\begin{equation*}
  \left\| (A-\id)(x) - y \right\|_X \le c\eta(\|x\|_X).
\end{equation*}
Consequently, the vector $\frac{(A-\id)(x) - y}{\eta(\|x\|_X)}$ has norm at most $c$, so it follows from \eqref{eq_oldlem3} that $(A-\id)(x) - y \ge - \eta(\|x\|_X) z$. Thus,%
\begin{equation*}
  (A-\id)(x) \ge -\eta(\|x\|_X)z + y \ge -\eta(\|x\|_X)z,%
\end{equation*}
which shows that \eqref{eq:uSGC-with-unit} fails for the function $\eta$, a contradiction.%

\ref{itm:MBI-criterion-with-unit-2} $\Rightarrow$ \ref{itm:MBI-criterion-with-unit-1}. Let \ref{itm:MBI-criterion-with-unit-2} hold with a certain $\eta\in\Kinf$. We show that \eqref{eq:uSGC-with-unit} holds for the function $\frac{\eta}{2\|z\|_X}$ substituted for $\eta$. Assume towards a contradiction that \eqref{eq:uSGC-with-unit} fails for the function $\frac{\eta}{2\|z\|_X}$. Then there is $x \in X^+ \setminus \{0\}$ such that%
\begin{equation*}
  (A - \id)(x) + \frac{\eta(\|x\|_X)}{2\|z\|_X} z \geq 0.%
\end{equation*}
Hence, it follows that%
\begin{align*}
  \dist\big((A - \id)(x),X^+\big) &\le \Bigl\|(A-\id)(x) - \big((A - \id)(x) + \frac{\eta(\|x\|_X)}{2\|z\|_X} z\big) \Bigr\|_X = \frac{\eta(\|x\|_X)}{2},%
\end{align*}
which shows that~\eqref{eq:uSGC-dist-form} fails for the function $\eta$. \qed%
\end{proof}

A typical example of an ordered Banach space whose cone has nonempty interior is $(X,X^+) = (\ell_{\infty}(I),\ell_{\infty}(I)^+)$ for some index set $I$. For instance, the vector ${\bf 1}$ is an interior point of the positive cone in this space.%

\subsection{Non-uniform small-gain conditions}\label{sec:Non-uniform small-gain conditions}

In Propositions~\ref{prop:criteria-MBI-without-unit} and~\ref{prop:criteria-MBI-with-unit}, we characterized the MBI property in terms of the uniform small-gain condition. In this {subsection}, we recall several further small-gain conditions, which have been used in the literature for the small-gain analysis of finite and infinite networks \cite{DRW07,DRW10,DMS19a}, and relate them to the uniform small-gain condition.%

In this {subsection}, we always suppose that $(X,X^+) = (\ell_{\infty}(I),\ell_{\infty}^+(I))$ for some nonempty index set $I$ (which is precisely the space in which gain operators act).

%

\begin{definition}\label{def:SGC}
We say that a nonlinear operator $A:\ell_{\infty}^+(I) \to \ell_{\infty}^+(I)$ satisfies%
\begin{enumerate}[label=(\roman*)]
\item \label{NL-SGC-def-item1} the \emph{small-gain condition} if%
\begin{equation}\label{eq:SGC}
   A(x)\not\geq x \mbox{\quad for all\ } x \in \ell_{\infty}^+(I)\setminus\{0\}.
\end{equation}
\item \label{NL-SGC-def-item2} the \emph{strong small-gain condition} if there exists $\rho\in\Kinf$ and a corresponding operator ${D_{\rho}}:\ell_{\infty}^+(I) \to \ell_{\infty}^+(I)$, defined for any $x\in \ell_{\infty}^+(I)$ by%
\begin{equation*}
  {D_{\rho}}(x) := \big((\id + \rho)(x_i)\big)_{i\in I},
\end{equation*}
such that%
\begin{equation}\label{eq:strong-SGC-nonlinear}
  {D_{\rho}}\circ A(x) \not\geq x \mbox{\quad for all\ } x\in \ell_{\infty}^+(I)\setminus\{0\}.%
\end{equation}

\item 
{ \label{NL-SGC-def-item3} the \emph{robust small-gain condition} if there is $\omega\in\Kinf$ with $\omega<\id$ such that for all $i,j \in I$ the operator $A_{i,j}$ given by%
\begin{equation}
\label{eq:A-modified}
  A_{i,j}(x) := A(x) + \omega(x_j) e_i \mbox{\quad for all\ } x \in \ell_{\infty}^+(I)%
\end{equation}
satisfies the small-gain condition \eqref{eq:SGC};
here, $e_i \in \ell_{\infty}(I)$ denotes the $i$-th canonical unit vector.

\item \label{NL-SGC-def-item4} the \emph{robust strong small-gain condition} if there are $\omega,\rho\in\Kinf$ with $\omega<\id$ such that for all $i,j \in I$ the operator
$A_{i,j}$ defined by \eqref{eq:A-modified} satisfies the strong small-gain condition \eqref{eq:strong-SGC-nonlinear} with the same $\rho$ for all $i,j$. \hfill $\lhd$
}
\end{enumerate}
\end{definition}


%

The strong small-gain condition was introduced in \cite{DRW07}, where it was shown that if the gain operator satisfies the strong small-gain condition, then a finite network consisting of ISS systems (defined in a summation formulation) is ISS. The robust strong small-gain condition has been introduced in \cite{DMS19a} in the context of the Lyapunov-based small-gain analysis of infinite networks.%

\begin{remark}\label{rem:Small-gain conditions and cycles} 
For finite networks, also so-called \emph{cyclic small-gain conditions} play an important role, as they help to effectively check the small-gain condition \eqref{eq:SGC} in the case when $A = \Gamma_\otimes$, which is important for the small-gain theorems in the maximum formulation, see \cite{Mir19b} for more discussions on this topic. For infinite networks, the cyclic condition for $\Gamma_\otimes$ is implied by \eqref{eq:SGC}, see \cite[Lem.~4.1]{DMS19a}, but is far too weak for the small-gain analysis. For {max-linear} systems, Remark~\ref{rem_homogen_spectral_radius} and Corollary~\ref{cor:SGT-sup-linear-gains} are reminiscent of the cyclic small-gain conditions.
\end{remark}

%
%
%
%

{

We say that a continuous function $\alpha:\R_+\to\R_+$ is of class $\PD$ if $\alpha(0)=0$ and $\alpha(r)>0$ for $r>0$.%

The following lemma is an extension of the considerations in \cite[p.~130]{KaJ11}.%


\begin{lemma}
\label{lem:KinfLipschitzLowerEstimate}
The following statements hold:%
\begin{enumerate}[label = (\roman*)]
	\item\label{itm:KinfLipschitzLowerEstimate-itm1} {For any $\alpha \in\PD$ and $L>0$, the function defined by%
\begin{equation}\label{eq:LowerEstimRho_def}
  \rho(r) := \inf_{y\geq 0} \big\{\alpha(y) + L|y-r|\big\}%
\end{equation}
is in $\PD$, satisfies $\rho(s) \leq \alpha(s)$ for all $s \in \R_+$, and is globally Lipschitz with Lipschitz constant $L$.}
	\item\label{itm:KinfLipschitzLowerEstimate-itm2} If in \ref{itm:KinfLipschitzLowerEstimate-itm1} $\alpha\in \K$, then $\rho$ given by \eqref{eq:LowerEstimRho_def} is a $\K$-function.%
	\item\label{itm:KinfLipschitzLowerEstimate-itm3} If in \ref{itm:KinfLipschitzLowerEstimate-itm1} $\alpha\in \Kinf$, then $\rho$ given by \eqref{eq:LowerEstimRho_def} is a $\Kinf$-function.%
\end{enumerate}
\end{lemma}

\begin{proof}
	%
\ref{itm:KinfLipschitzLowerEstimate-itm1}. Consider $\rho$ given by \eqref{eq:LowerEstimRho_def}. Note that for any $r>0$ it holds that $\alpha(y) + L|y-r|\to \infty$ as $y \to\infty$. Thus, there is $r^*>0$ such that $\rho(r) = \inf_{y\in[0,r^*]} \big\{\alpha(y) + L|y-r|\big\}$, and as $\alpha$ is continuous, there is $y^*=y^*(r)$ such that $\rho(r) = \alpha(y^*) + L|y^*-r|$.%

Clearly, $0\leq \rho(r)\leq \alpha(r)$ for all $r\geq 0$. Assume that $\rho(r)= 0$ for some $r\geq 0$. By the above argument, $\rho(r) = \alpha(r)$, and as $\alpha(r)=0$ if and only if $r=0$, it follows that $\rho(0)=0$ and $\rho(r)>0$ for $r>0$.%

Next, for any $r_1,r_2 \geq 0$ we have by the triangle inequality%
\begin{eqnarray*}
\rho(r_1)-\rho(r_2) 
&=& \inf_{y\geq 0} \big\{\alpha(y) + L|y-r_1|\big\} - \inf_{y\geq 0} \big\{\alpha(y) + L|y-r_2|\big\}\\
&\leq& \inf_{y\geq 0} \big\{\alpha(y) + L|y-r_2| + L|r_2-r_1|\big\} - \inf_{y\geq 0} \big\{\alpha(y) + L|y-r_2|\big\}\\
&=& L|r_2-r_1|.
\end{eqnarray*}
Similarly, using the triangle inequality for the second term, we obtain%
\begin{equation*}
  \rho(r_1)-\rho(r_2) \geq -L|r_2-r_1|,%
\end{equation*}
and thus $\rho$ is globally Lipschitz with Lipschitz constant $L$, and is of class $\PD$.%

\ref{itm:KinfLipschitzLowerEstimate-itm2}. Let $\alpha\in\K$. Pick any $r_1,r_2 \geq 0$ with $r_1 >r_2$ and let $y_1 \geq 0$ be so that $\rho(r_1) = \alpha(y_1) + |y_1 - r_1|$. Consider the expression%
\begin{equation}
\label{eq:Kinf-lemma-lower-estimate-tmp}
\rho(r_1) - \rho(r_2) = \alpha(y_1) + L|y_1-r_1| - \inf_{y\geq 0} \{\alpha(y) + L|y-r_2|\}.
\end{equation}
If $y_1\geq r_2$, then $\rho(r_1) - \rho(r_2) \geq \alpha(y_1) + L|y_1-r_1|  -\alpha(r_2) >0$, as
 $\alpha$ is increasing.

If $y_1<r_2$, then 
\begin{eqnarray*}
\rho(r_1) - \rho(r_2) &\geq& \alpha(y_1) + L|y_1-r_1| - \big(\alpha(y_1) + L|y_1 - r_2|\big) = L(r_1-r_2)>0.
\end{eqnarray*}


\ref{itm:KinfLipschitzLowerEstimate-itm3}. Let $\alpha\in\Kinf$.
Assume to the contrary that $\rho$ is bounded: $\rho(r) \leq M$ for all $r$. Then for every $r$ there is $r'$ with $\alpha(r') + L|r - r'| \leq 2M$. Looking at the second term, we see that $r' \rightarrow \infty$ as $r \rightarrow \infty$. But then $\alpha(r') \rightarrow \infty$, a contradiction. \qed
\end{proof}

Items (ii) and (iii) of the following elementary lemma are variations of \cite[Lem.~1.1.5]{Rue07} and \cite[Lem.~1.1.3, item 1]{Rue07}, respectively.%

\begin{lemma}\label{lem:uSGC-lemma3}
\begin{enumerate}
\item[(i)] For any $\alpha\in\Kinf$ there is $\eta\in\Kinf$ such that $\eta(r)\leq\alpha(r)$ for all $r\geq 0$, and $\id-\eta\in\Kinf$.
\item[(ii)] For any $\eta\in\Kinf$ with $\id - \eta\in\Kinf$, there is $\rho\in\Kinf$ such that $(\id-\eta)^{-1} = \id + \rho$.%
\item[(iii)] For any $\eta\in\Kinf$ such that $\id - \eta\in\Kinf$ there are $\eta_1,\eta_2\in\Kinf$ such that $\id-\eta_1,\id-\eta_2\in\Kinf$ and $\id-\eta = (\id-\eta_1) \circ (\id-\eta_2)$.
\end{enumerate}
\end{lemma}
%
	%

\begin{proof}
(i) Take any $L\in(0,1)$ and construct $\rho\in\Kinf$, globally Lipschitz with\linebreak Lipschitz constant $L$ as in Lemma~\ref{lem:KinfLipschitzLowerEstimate}. Clearly, $(\id - \rho)(0) = 0$, and $\id - \rho$ is continuous. For $r,s\geq 0$ with $r>s$, we have%
\begin{align*}
  r-\rho(r) - (s-\rho(s)) &= r-s - (\rho(r)-\rho(s))\geq r-s -L(r-s) \\
	&= (1-L)(r-s)>0,%
\end{align*}
and thus $\id-\rho$ is increasing. 
Furthermore, $r-\rho(r)\geq (1-L)r\to \infty$ as $r\to\infty$, and thus $\id-\rho\in\Kinf$.%

(ii) Define $\rho:=\eta\circ(\id-\eta)^{-1}$. As $\rho$ is a composition of $\Kinf$-functions, $\rho\in\Kinf$. It holds that 
$(\id + \rho)\circ(\id - \eta)= \id - \eta + \eta\circ(\id-\eta)^{-1}\circ(\id - \eta) = \id - \eta + \eta = \id$, and thus $\id + \rho = (\id-\eta)^{-1}$.%

(iii) Choose $\eta_2:=\frac{1}{2}\eta$ and $\eta_1:=\frac{1}{2}\eta\circ (\id-\eta_2)^{-1}$. A direct calculation shows the claim. \qed
\end{proof}
}

Now we give a criterion for the robust strong small-gain condition.%

\begin{proposition}\label{prop:Criterion-robust-strong-SGC} 
A nonlinear operator $A:\ell_{\infty}^+(I) \to \ell_{\infty}^+(I)$ satisfies the robust strong small-gain condition if and only if there are $\omega,\eta\in\Kinf$ and an operator $\vec{\eta}:\ell_{\infty}^+(I) \to \ell_{\infty}^+(I)$, defined by%
\begin{equation}\label{eq:vec-eta}
  \vec{\eta}(x) := (\eta(x_i))_{i\in I} \mbox{\quad for all\ } x\in \ell_{\infty}^+(I),%
\end{equation}
such that for all $k\in I$ it holds that%
\begin{equation}\label{eq:rsSGC-l-infty}
  A(x)\not\geq x - \vec{\eta}(x) - \omega(\|x\|_{\ell_{\infty}(I)}) e_k \mbox{\quad for all\ } x\in \ell_{\infty}^+(I)\setminus\{0\}.%
\end{equation}
\end{proposition}

\begin{proof}
\q{$\Rightarrow$}: Let the robust strong small-gain condition hold with corresponding {$\rho,\omega$ and $D_{\rho}$}. Then for any $x = (x_i)_{i\in I} \in \ell_{\infty}^+(I)\setminus\{0\}$ and any $j,k\in I$, it holds that%
\begin{equation}\label{eq:rsSGC-l-infty-componentwise-1}
  \exists i\in I:\quad \big[{D_{\rho}}\big(A(x) + \omega(x_j) e_k\big)\big]_i = (\id + \rho)\big([A(x) + \omega(x_j) e_k]_i\big) < x_i.%
\end{equation}
As $\rho\in\Kinf$, there is $\eta\in\Kinf$ such that $\id-\eta = (\id + \rho)^{-1}\in\Kinf$, which can be shown as in {Lemma~\ref{lem:uSGC-lemma3}(ii)}. Thus, \eqref{eq:rsSGC-l-infty-componentwise-1} is equivalent to%
\begin{equation}\label{eq:rsSGC-l-infty-componentwise}
  \exists i\in I:\quad A(x)_i < x_i - \eta(x_i) - \big[\omega( x_j) e_k\big]_i.%
\end{equation}
As for each $x \in \ell_{\infty}^+(I)$ there is $j\in I$ such that $x_j \geq \frac{1}{2}\|x\|_{\ell_{\infty}(I)}$, the condition \eqref{eq:rsSGC-l-infty-componentwise} with this particular $j$ implies that%
\begin{align*}
\exists i\in I:\quad A(x)_i &< x_i - \eta(x_i) - \Big[\omega\big(\frac{1}{2} \|x\|_{\ell_{\infty}(I)}\big) e_k\Big]_i
													= \Big[x - \vec{\eta}(x) - \omega\big(\frac{1}{2} \|x\|_{\ell_{\infty}(I)}\big) e_k\Big]_i,%
\end{align*}
which is up to the constant the same as \eqref{eq:rsSGC-l-infty}.%

\q{$\Leftarrow$}: Let \eqref{eq:rsSGC-l-infty} hold with a certain $\eta_1\in\Kinf$ and a corresponding $\vec{\eta}_1$. By {Lemma~\ref{lem:uSGC-lemma3}(i)}, one can choose $\eta\in\Kinf$, such that $\eta\leq \eta_1$ and $\id-\eta\in\Kinf$. Then \eqref{eq:rsSGC-l-infty} holds with this $\eta$ and a corresponding $\vec{\eta}$, i.e., for all $k\in I$ we have%
\begin{equation*}
  \exists i\in I:\quad A(x)_i < x_i - \eta(x_i) - \big[\omega( \|x\|_{\ell_{\infty}(I)}) e_k\big]_i.%
\end{equation*}
As $\|x\|_{\ell_{\infty}(I)}\geq x_j$ for any $j\in I$, this implies that for all $j,k\in I$ it holds that%
\begin{equation*}
  \exists i\in I:\quad A(x)_i < x_i - \eta(x_i) - \big[\omega( x_j) e_k\big]_i,%
\end{equation*}
and thus%
\begin{equation*}
  \exists i\in I:\quad \big[ A(x) + \omega( x_j) e_k\big]_i < (\id-\eta)(x_i).%
\end{equation*}
As $\eta\in\Kinf$ satisfies $\id-\eta\in\Kinf$, by {Lemma~\ref{lem:uSGC-lemma3}(ii)} there is $\rho\in\Kinf$ such that $(\id-\eta)^{-1} = \id + \rho$, and thus for all $j,k\in I$ property \eqref{eq:rsSGC-l-infty-componentwise-1} holds, which shows that $A$ satisfies the robust strong small-gain condition. \qed%
\end{proof}

Specialized to the strong small-gain condition, Proposition~\ref{prop:Criterion-robust-strong-SGC} reads as follows.%

\begin{corollary}\label{cor:Criterion-strong-SGC} 
A nonlinear operator $A:\ell_{\infty}^+(I) \to \ell_{\infty}^+(I)$ satisfies the strong small-gain condition if and only if there are $\eta\in\Kinf$ and an operator $\vec{\eta}:\ell_{\infty}^+(I) \to \ell_{\infty}^+(I)$, defined via \eqref{eq:vec-eta} such that%
\begin{equation*}
  A(x)\not\geq x - \vec{\eta}(x) \mbox{\quad for all\ } x\in \ell_{\infty}^+(I)\setminus\{0\}.%
\end{equation*}
\end{corollary}

The next proposition shows that the uniform small-gain condition is at least not weaker than the robust strong small-gain condition.%

{
\begin{proposition}\label{prop:uSGC-implies-sSGC} 
Let $A:\ell_{\infty}^+(I)\to \ell_{\infty}^+(I)$ be a nonlinear operator. If $A$ satisfies the uniform small-gain condition, then $A$ satisfies the robust strong small-gain condition.
\end{proposition}

\begin{proof}
As $A$ satisfies the uniform small-gain condition with $\eta$, from the proof of Proposition~\ref{prop:criteria-MBI-with-unit} with $z:= \bf 1$, we see that for all $x \in \ell_{\infty}^+(I) \setminus \{0\}$%
\begin{equation*}
  A(x) \not\geq x - \frac{1}{2\|{\bf 1}\|_{\ell_{\infty}(I)}}\eta(\|x\|_{\ell_{\infty}(I)}){\bf 1} 
	= x - \frac{1}{2}\eta(\|x\|_{\ell_{\infty}(I)}){\bf 1}.%
\end{equation*}
For any $x\in \ell_{\infty}^+(I)$ and any $k\in I$, it holds that%
\begin{align*}
\frac{1}{2}\eta(\|x\|_{\ell_{\infty}(I)}){\bf 1}
&=\frac{1}{4}\eta(\|x\|_{\ell_{\infty}(I)}){\bf 1} + \frac{1}{4}\eta(\|x\|_{\ell_{\infty}(I)}){\bf 1}\\
&\ge \frac{1}{4}\vec{\eta}(x) + \frac{1}{4}\eta(\|x\|_{\ell_{\infty}(I)})e_k,%
\end{align*}
and by Proposition~\ref{prop:Criterion-robust-strong-SGC}, $A$ satisfies the robust strong small-gain condition. \qed%
\end{proof}
}

%

\subsection{The finite-dimensional case}\label{sec:Finite-dimensional systems}

The case of a finite-dimensional $X$ is particularly important as it is a key to the stability analysis of finite networks.%

\begin{proposition}
\label{prop:small-gain-condition-n-dim-general}
Assume that $(X,X^+) = (\R^n,\R^n_+)$ for some $n\in\N$, where $\R^n$ is equipped with the maximum norm $\|\cdot\|$ and $\R^n_+$ denotes the standard positive cone in $\R^n$. Further assume that the operator $A$ is continuous and monotone. Then the following statements are equivalent:%
\begin{enumerate}[label = (\roman*)]
	\item\label{itm:n-dim-small-gain-criterion-1} System \eqref{eq_monotone_system} has the MLIM property.%
	\item\label{itm:n-dim-small-gain-criterion-2} The operator $\id - A$ has the MBI property.%
	\item\label{itm:n-dim-small-gain-criterion-3} The uniform small-gain condition holds: There is an $\eta \in \Kinf$ such that $\dist(A(x) - x,X^+) \geq \eta(\|x\|)$ for all $x \in X^+$.%
	\item\label{itm:n-dim-small-gain-criterion-4} There is an $\eta \in \Kinf$ such that%
\begin{equation*}
  A(x) \not\geq x - \eta(\|x\|){\bf 1} \mbox{\quad for all\ } x \in X^+ \setminus \{0\}.%
\end{equation*}
\end{enumerate}

Additionally, if $A$ is either $\Gamma_\boxplus$ or $\Gamma_\otimes$, then the above conditions are equivalent to%

\begin{enumerate}[label = (\roman*), start = 5]
	\item\label{itm:n-dim-small-gain-criterion-5} $A$ satisfies the robust strong small-gain condition.
	\item\label{itm:n-dim-small-gain-criterion-6} $A$ satisfies the strong small-gain condition.
\end{enumerate}
\end{proposition}

\begin{proof}
\ref{itm:n-dim-small-gain-criterion-1} $\Rightarrow$ \ref{itm:n-dim-small-gain-criterion-2}. Follows from Proposition~\ref{prop_mlim_implies_mbip}.

\ref{itm:n-dim-small-gain-criterion-2} $\Iff$ \ref{itm:n-dim-small-gain-criterion-3} $\Iff$ \ref{itm:n-dim-small-gain-criterion-4}. 
Follows from Propositions~\ref{prop:criteria-MBI-with-unit},~\ref{prop:criteria-MBI-without-unit}.

\ref{itm:n-dim-small-gain-criterion-2} $\Rightarrow$ \ref{itm:n-dim-small-gain-criterion-1}. This follows from Proposition \ref{prop_compact_operators} since the cone $\R^n_+$ has the Levi property.

\ref{itm:n-dim-small-gain-criterion-4} $\Rightarrow$ \ref{itm:n-dim-small-gain-criterion-5}. Follows by Proposition~\ref{prop:uSGC-implies-sSGC}. 

\ref{itm:n-dim-small-gain-criterion-5} $\Rightarrow$ \ref{itm:n-dim-small-gain-criterion-6}. Clear.

\ref{itm:n-dim-small-gain-criterion-6} $\Rightarrow$ \ref{itm:n-dim-small-gain-criterion-2}. Follows by \cite[Thm.~6.1]{Rue10}. \qed%
\end{proof}

\begin{remark}
\label{rem:MAFs} 
The class of operators for which the equivalence between \ref{itm:n-dim-small-gain-criterion-1}--\ref{itm:n-dim-small-gain-criterion-4} and \ref{itm:n-dim-small-gain-criterion-5}, \ref{itm:n-dim-small-gain-criterion-6} can be shown, can be made considerably larger using the monotone aggregation functions formalism, see \cite[Thm.~6.1]{Rue10}. However, the proof of this implication in \cite[Lem.~13]{DRW07} uses more structure of the gain operator than merely monotonicity. Thus, the question if this implication is valid for general monotone $A$ is still open.
\end{remark}

{
\begin{remark}
\label{rem:IOS} 
An interesting research direction could be the development of the small-gain theorems for the case when the subsystems obtain the outputs of other subsystems, instead of their full states, as inputs (so-called IOS small-gain theorems). For finite networks, such trajectory-based results have been reported in \cite{JTP94} for couplings of two systems, and in \cite{Rue07,JiW08} for any finite number of finite-dimensional systems. The authors are not aware of such trajectory-based results for networks with infinite-dimensional components and/or infinite networks.
\end{remark}
}

\subsection{Systems with linear gains}
\label{sec:Systems with linear gains}

Here we show that in the case of linear and sup-linear gain operators the MBI and MLIM properties are equivalent and can be characterized via the spectral condition.
%


\begin{definition}
\label{def:eISS-discrete-time-inequalities}
Let $(X,X^+)$ be an ordered Banach space. 
System \eqref{eq_monotone_system} is \emph{exponentially input-to-state stable (eISS)} if there are $M\geq 1$, $a\in(0,1)$ and $\gamma \in \Kinf$ such that for every $u \in \ell_{\infty}(\Z_+,X^+)$ and any solution $x(\cdot) = (x(k))_{k\in\Z_+}$ of \eqref{eq_monotone_system} it holds that%
\begin{equation}
\label{eq:eISS-inequalities}
  \|x(k)\|_X \leq M\|x(0)\|_X a^k + \gamma(\|u\|_{\infty}) \mbox{\quad for all\ } k \in \Z_+.%
\end{equation}
\end{definition}

{
For linear systems, we obtain the following result, that we use to formulate an efficient small-gain theorem in summation formulation, see Corollary \ref{cor:SGT-sum-linear-gains}.
\begin{proposition}
\label{prop:eISS-criterion-linear-systems}
Let $(X,X^+)$ be an ordered Banach space with a generating and normal cone $X^+$. 
{Let the operator $A:X^+ \to X^+$ be the restriction to $X^+$ of a positive linear operator on $X$.} Then the following statements are equivalent:%
\begin{enumerate}[label = (\roman*)]
\item\label{itm:linear-LIM-criterion-1} System \eqref{eq_monotone_system} is exponentially ISS.
\item\label{itm:linear-LIM-criterion-3} System \eqref{eq_monotone_system} satisfies the MLIM property.%
\item\label{itm:linear-LIM-criterion-4} The operator $\id-A$ satisfies the MBI property.%
\item\label{itm:linear-LIM-criterion-5} The spectral radius of $A$ satisfies $r(A) < 1$.%
\end{enumerate}
\end{proposition}

\begin{proof}
The implication ``\ref{itm:linear-LIM-criterion-1} $\Rightarrow$ \ref{itm:linear-LIM-criterion-3}'' is trivial. 
By Proposition \ref{prop_mlim_implies_mbip}, \ref{itm:linear-LIM-criterion-3} implies \ref{itm:linear-LIM-criterion-4}. 

\ref{itm:linear-LIM-criterion-4} $\Rightarrow$ \ref{itm:linear-LIM-criterion-5}. It is easy to check that if $A$ is homogeneous of degree one  and $\id - A$ satisfies the MBI property with a certain $\xi\in\Kinf$, then $\id - A$ satisfies the MBI property with $r\mapsto \xi(1)r$ instead of $\xi$. 
The application of \cite[Thm.~3.3]{GlM20} shows \ref{itm:linear-LIM-criterion-5}.

%
%
%

\ref{itm:linear-LIM-criterion-5} $\Rightarrow$ \ref{itm:linear-LIM-criterion-1}. Follows from Proposition~\ref{prop:eISS-criterion-homogeneous-systems}.
\qed%
\end{proof}

For sup-linear systems, MBI is again equivalent to eISS, and the following holds:


\begin{proposition}
\label{prop:join-morphism-eISS-criterion}
Assume that the gains $\gamma_{ij}$, $(i,j) \in I^2$, are all linear and that the associated gain operator $\Gamma_{\otimes}$ is well-defined. Then the following statements are equivalent:%
\begin{enumerate}[label = (\roman*)]
\item\label{itm:join-morphism-eISS-criterion-1} The operator $\id - \Gamma_{\otimes}$ satisfies the MBI property.%
\item\label{itm:join-morphism-eISS-criterion-2} There are $\lambda\in(0,1)$ and $s_0 \in \inner(\ell_{\infty}^+(I))$ such that 
\begin{eqnarray}
\Gamma_{\otimes}(s_0) \leq \lambda s_0.
\label{eq:Point-of-strict-decay}
\end{eqnarray}

\item\label{itm:join-morphism-eISS-criterion-3} The spectral radius of $\Gamma_{\otimes}:\ell_{\infty}^+(I) \to \ell_{\infty}^+(I)$ satisfies%
\begin{equation*}
  r(\Gamma_{\otimes}) = \lim_{n\rightarrow\infty} \sup_{s \in \ell_{\infty}^+(I) ,\; \|s\|_{\ell_{\infty}} = 1} \|\Gamma_{\otimes}^n(s)\|_{\ell_{\infty}(I)}^{1/n}
	= \lim_{n\rightarrow\infty} \|\Gamma_{\otimes}^n({\bf 1})\|_{\ell_{\infty}(I)}^{1/n} < 1.%
\end{equation*}
\item\label{itm:join-morphism-eISS-criterion-4} The system \eqref{eq_monotone_system} with $A = \Gamma_{\otimes}$ is eISS.%
\item\label{itm:join-morphism-eISS-criterion-5} The system \eqref{eq_monotone_system} with $A = \Gamma_{\otimes}$ has the MLIM property.%
\end{enumerate}
\end{proposition}

\begin{proof}
By Proposition~\ref{prop:eISS-criterion-homogeneous-systems}, \ref{itm:join-morphism-eISS-criterion-3} is equivalent to \ref{itm:join-morphism-eISS-criterion-4}. Clearly, \ref{itm:join-morphism-eISS-criterion-4} implies \ref{itm:join-morphism-eISS-criterion-5}. By Proposition~\ref{prop_mlim_implies_mbip}, \ref{itm:join-morphism-eISS-criterion-5} implies \ref{itm:join-morphism-eISS-criterion-1}.%

\ref{itm:join-morphism-eISS-criterion-1} $\Rightarrow$ \ref{itm:join-morphism-eISS-criterion-2}. By Proposition \ref{prop:criteria-MBI-without-unit}, the MBI property of $\id - \Gamma_{\otimes}$ is equivalent to the uniform small-gain condition. Then Proposition \ref{prop:criteria-MBI-with-unit} shows that%
\begin{equation*}
  \Gamma_{\otimes}(s) \not\geq s - \eta(\|s\|_{\ell_{\infty}}){\bf 1} \mbox{\quad for all\ } s \in \ell_{\infty}^+(I) \setminus \{0\}%
\end{equation*}
for some $\eta \in \Kinf$. In particular,%
\begin{equation*}
  \Gamma_{\otimes}\Bigl(\frac{s}{\|s\|_{\ell_{\infty}}}\Bigr) \not\geq \frac{s}{\|s\|_{\ell_{\infty}}} - \eta(1){\bf 1} \mbox{\quad for all\ } s \in \ell_{\infty}^+(I) \setminus \{0\}.%
\end{equation*}
Multiplying this inequality by $\|s\|_{\ell_{\infty}}$, putting $\eta := \eta(1)$ and using the homogeneity of degree one of $\Gamma_{\otimes}$ yields%
\begin{equation*}
  \Gamma_{\otimes}(s) \not\geq s - \eta\|s\|_{\ell_{\infty}} {\bf 1} \mbox{\quad for all\ } s \in \ell_{\infty}^+(I) \setminus \{0\}.%
\end{equation*}
Then for any $s \in \ell_{\infty}^+(I)$ we have%
\begin{align*}
  (1 + \ep)\Gamma_{\otimes}(s) &\not\geq (1 + \ep)(s - \eta \|s\|_{\ell_{\infty}} {\bf 1}) = s + \ep s - (1 + \ep) \eta \|s\|_{\ell_{\infty}} {\bf 1}.
\end{align*}
As $s + \ep s - (1 + \ep) \eta \|s\|_{\ell_{\infty}} {\bf 1} \leq s - [(1 + \ep)\eta - \ep] \|s\|_{\ell_{\infty}}{\bf 1}$,
we have
\begin{align*}
  (1 + \ep)\Gamma_{\otimes}(s) &\not\geq s - [(1 + \ep)\eta - \ep] \|s\|_{\ell_{\infty}}{\bf 1}.
\end{align*}
Choosing $\varepsilon>0$ small enough, Proposition \ref{prop:uSGC-implies-sSGC} implies that $(1 + \ep)\Gamma_{\otimes}$ satisfies the robust strong small-gain condition. By Lemma \ref{lem_Q}, the operator%
\begin{equation*}
  Q^{\ep}(s) := \sup_{k \in \Z_+} (1 + \ep)^k \Gamma_{\otimes}^k(s) \mbox{\quad for all\ } s \in \ell_{\infty}^+(I)%
\end{equation*}
is well-defined and satisfies%
\begin{equation*}
  \Gamma_{\otimes}(Q^{\ep}(s)) \leq \frac{1}{1 + \ep}Q^{\ep}(s) \mbox{\quad for all\ } s \in \ell_{\infty}^+(I).
\end{equation*}
In particular, this holds for $s = {\bf 1}$. Since $s_0 := Q^{\ep}({\bf 1}) \geq {\bf 1}$, we have $s_0 \in \inner(\ell_{\infty}^+(I))$. 

\ref{itm:join-morphism-eISS-criterion-2} $\Rightarrow$ \ref{itm:join-morphism-eISS-criterion-3}. By monotonicity and homogeneity of degree one of $\Gamma_{\otimes}$, we have%
\begin{equation*}
  \Gamma_{\otimes}^k(s_0) \leq \lambda^k s_0 \mbox{\quad for all\ } k \geq 1.%
\end{equation*}
There exists $n \in \N$ such that any $s \in \ell_{\infty}^+(I)$ with $\|s\|_{\ell_{\infty}} = 1$ satisfies $s \leq ns_0$. Hence,%
\begin{equation*}
  \Gamma_{\otimes}^k(s) \leq \Gamma_{\otimes}^k(ns_0) = n \Gamma_{\otimes}^k(s_0) \leq n \lambda^k s_0 \mbox{\quad for all\ } k \geq 1,\ \|s\|_{\ell_{\infty}} = 1.%
\end{equation*}
This implies $r(\Gamma_{\otimes}) \leq \lambda < 1$, which completes the proof. \qed%
\end{proof}

\begin{remark}\label{rem:Paths-of-strict-decay} 
The special form of the operator $\Gamma_\otimes$ is used in Proposition~\ref{prop:join-morphism-eISS-criterion} only for the proof of 
the implication \ref{itm:join-morphism-eISS-criterion-1} $\Rightarrow$ \ref{itm:join-morphism-eISS-criterion-2}. The remaining implications are valid for considerably more general types of operators. Note that if $s_0$ is as in item \ref{itm:join-morphism-eISS-criterion-2}, then $ts_0$ also satisfies all conditions in item \ref{itm:join-morphism-eISS-criterion-2}, for any $t>0$ and thus we can construct a path of strict decay $t\mapsto t s_0$ for the gain operator $\Gamma_\otimes$, which is an important ingredient for the proof of the Lyapunov-based ISS small-gain theorem, see \cite{DRW06b}.
\end{remark}

}

\begin{acknowledgements}
A.~Mironchenko is supported by the German Research Foundation (DFG) via the grant MI 1886/2-1.
C.~Kawan is supported by the DFG through the grant ZA 873/4-1.
\end{acknowledgements}

\bibliographystyle{abbrv}
\bibliography{Mir_LitList_NoMir,MyPublications}

\begin{appendices}

\setcounter{section}{0}

\section{Exponential ISS of linear and homogeneous of degree one subadditive discrete-time systems}
\label{sec:Exponential ISS of discrete-time systems}

Here we characterize  exponential ISS for homogeneous of degree one and subadditive operators.

\begin{proposition}\label{prop:eISS-criterion-homogeneous-systems}
Let $(X,X^+)$ be an ordered Banach space with a generating and normal cone $X^+$. Consider system \eqref{eq_monotone_system} and assume that the operator $A:X^+\to X^+$ is monotone and satisfies the following properties:%
\begin{enumerate}[label = (\roman*)]
\item\label{itm:eISS-criterion-homog-1} $A$ is homogeneous of degree one, i.e., $A(rx) = rA(x)$ for all $x\in X^+$ and $r \geq 0$.%
\item\label{itm:eISS-criterion-homog-2} $A$ is subadditive, i.e., $A(x + y) \leq A(x) + A(y)$ for all $x,y\in X^+$.%
\item\label{itm:eISS-criterion-homog-3}  $A$ satisfies%
\begin{equation*}
  C := \sup_{x\in X^+ ,\; \|x\|_X=1}\|A(x)\|_X < \infty.%
\end{equation*}
\end{enumerate}
{Then $A$ is globally Lipschitz continuous and the following statements are equivalent:
\begin{enumerate}[label = (\alph*)]
	\item\label{itm:eISS-criterion-equiv-1} System \eqref{eq_monotone_system} is eISS. 

	\item\label{itm:eISS-criterion-equiv-2}  It holds that 
	\begin{equation}\label{eq_def_spectralradius}
  r(A) := \lim_{n \rightarrow \infty} \sup_{x \in X^+,\; \|x\|_X = 1}\|A^n(x)\|_X^{1/n} < 1.%
\end{equation}
\item\label{itm:eISS-criterion-equiv-3} There is a globally Lipschitz $V:X^+\to\R_+$ and $L_1,L_2,\psi>0$, $\eta>1$, such that 
\begin{equation}
\label{eq:Sandwich}
  L_1\|x\|_X \leq V(x) \leq L_2 \|x\|_X,\quad x \in X^+,
\end{equation}
%
%
and for any $u \in \ell_{\infty}(\Z_+,X^+)$, and any solution of \eqref{eq_monotone_system} it holds that 
\begin{equation}
\label{eq:Dissipative-inequality}
  V(x(k+1)) \leq \eta^{-1} V(x(k)) + \psi \|u\|_{\infty} \mbox{\quad for all\ } k \geq 0.%
\end{equation}
\end{enumerate}
}
%
\end{proposition}

\begin{proof}
{
\textbf{$A$ is Lipschitz continuous.}
Pick any $x,y \in X^+$. As $X^+$ is generating, there are $M>0$ (which does not depend on $x,y$) and $a,b\in X^+$ such that $x-y = a-b$ and 
$\|a\|_X\leq M\|x-y\|_X$, $\|b\|_X\leq M\|x-y\|_X$.
Hence, for all $x,y\in X^+$ we have 
\begin{align*}
A(x) - A(y)
 &=  A(x-y + y)  -A(y)  =  A(a-b + y)  -A(y)\\
 &\leq  A(a + y)  -A(y) \leq  A(a) + A(y)   - A(y) = A(a).
\label{eq:A(x)-A(y)-upper}
\end{align*}
Analogously, we obtain for all $x,y\in X^+$ that $A(x) - A(y) \geq - A(b)$. As $X^+$ is normal, due to \cite[Thm.~2.38]{AlT07}, there is $c>0$, depending only on $(X,X^+)$, such that 
\begin{align*}
\|A(x) - & A(y)\|_{X} 
\leq c\max\{\|A(a)\|_{X},\|A(b)\|_{X}\}\\
\leq& c\max\{\|a\|_X\|A(a/\|a\|_X)\|_{X},\|b\|_X\|A(b/\|b\|_X)\|_{X}\}
\leq cCM \|x-y\|_{X}.
\end{align*}
}
\q{\ref{itm:eISS-criterion-equiv-1} $\Rightarrow$ \ref{itm:eISS-criterion-equiv-2}}: If \eqref{eq_monotone_system} is eISS, then for $u\equiv 0$, any $x \in X^+$ and for the solution $x(k+1) = A(x(k))$ of  \eqref{eq_monotone_system}, the inequality \eqref{eq:eISS-inequalities} implies that $\|A^n(x)\|_X\leq Ma^n\|x\|_X$ for all $n\in\Z_+$.
Hence, $\sup_{x \in X^+,\; \|x\|_X = 1}\|A^n(x)\|_X^{1/n} \leq M^{1/n}a \to a$ as $n\to\infty$ and thus $r(A) \leq a<1$.%

\q{\ref{itm:eISS-criterion-equiv-2} $\Rightarrow$ \ref{itm:eISS-criterion-equiv-3}}: 
 From the assumptions \ref{itm:eISS-criterion-homog-1} and \ref{itm:eISS-criterion-homog-3} together it follows that%
\begin{equation}
\label{eq_abound}
  \|A(x)\|_X = \|x\|_X \|A(x/\|x\|_X)\|_X \leq C \|x\|_X \mbox{\quad for all\ } x \in X^+ \setminus \{0\}.%
\end{equation}
Consider the sequence%
\begin{equation*}
  a_n := \sup_{x \in X^+,\; \|x\|_X = 1}\|A^n(x)\|_X,\quad n \in \Z_+.%
\end{equation*}
This sequence is submultiplicative, as for all $m,n\in\Z_+$ it holds that
\begin{align*}
  a_{n+m} &= \sup_{x \in X^+,\; \|x\|_X = 1}\|A^m(A^n(x))\|_X = \sup_{x \in X^+ ,\; \|x\|_X = 1} \|A^n(x)\|_X \Bigl\|A^m\Bigl(\frac{A^n(x)}{\|A^n(x)\|_X}\Bigr)\Bigr\|_X \\
	        &\leq \sup_{x \in X^+ ,\; \|x\|_X = 1} \|A^n(x)\|_X \cdot \sup_{x \in X^+ ,\; \|x\|_X=1} \|A^m(x)\|_X = a_n \cdot a_m.%
\end{align*}
By a submultiplicative version of the Fekete's subadditive lemma, $\lim_{n\to\infty}a_n^{\frac{1}{n}} = \inf_{n\to\infty}a_n^{\frac{1}{n}}\leq a_1<\infty$, and thus the limit in \eqref{eq_def_spectralradius} exists.%

We fix $\eta > 1$ such that $\eta r(A) < 1$ and define a function $V:X^+ \rightarrow \R_+$ by%
\begin{equation*}
  V(x) := \sup_{n \in \Z_+} \eta^n \|A^n(x)\|_X \mbox{\quad for all\ } x \in X^+.%
\end{equation*}
Setting $n:=0$ in the supremum, we see that $\|x\|_X\leq V(x)$ for all $x \in X^+$. Since $r(A) < \eta^{-1}$, there exists $N \in \N$ so that%
\begin{equation*}
  \sup_{x \in X^+ ,\; \|x\|_X = 1}\|A^n(x)\|_X \leq \eta^{-n} \mbox{\quad for all\ } n \geq N.%
\end{equation*}
By homogeneity of degree one of $A$, this implies%
\begin{equation*}
  \eta^n \|A^n(x)\|_X = \|x\|_X \eta^n \|A^n(\tfrac{x}{\|x\|_X})\|_X \leq \|x\|_X \mbox{\quad for all\ } n \geq N,\ x \in X^+ \setminus \{0\}.%
\end{equation*}
By \eqref{eq_abound}, we have $\|A(x)\|_X \leq C\|x\|_X$ for all $x \in X^+$. Due to homogeneity of $A$%
\begin{equation*}
  \|A^n(x)\|_X = \|A^{n-1}(x)\|_X \|A(\frac{A^{n-1}(x)}{\|A^{n-1}(x)\|_X})\|_X \leq C\|A^{n-1}(x)\|_X%
\end{equation*}
for all $x \in X^+$, and by induction $\|A^n(x)\|_X \leq C^n\|x\|_X$ for all $x \in X^+$.%

Since $\eta^0 \|A^0(x)\|_X=\|x\|_X$, with $\psi := \max_{0 \leq n < N} (\eta C)^n$ we have%
\begin{equation}\label{eq:LF-estimate-above}
  V(x) = \sup_{n \in \Z_+} \eta^n \|A^n(x)\|_X = \sup_{0\leq n < N} \eta^n \|A^n(x)\| \leq \psi \|x\|_X.%
\end{equation}
Also observe that%
\begin{equation*}
  V(A(x)) = \sup_{n \in \Z_+} \eta^n \|A^{n+1}(x)\|_X = \eta^{-1} \sup_{n \in \Z_+} \eta^{n+1} \|A^{n+1}(x)\|_X \leq \eta^{-1} V(x).%
\end{equation*}
As $A$ is monotone and subadditive, it holds by induction for all $n\in\N$ that 
\begin{equation*}
  A^n(x+y) = A^{n-1}(A(x+y)) \leq A^{n-1}(A(x)+A(y)) \leq A^n(x) + A^n(y),%
\end{equation*}
that is, $A^n$ are subadditive as well.%

We can assume without loss of generality that the norm $\|\cdot\|_X$ is monotone, i.e., $0 \leq x \leq y$ implies $\|x\|_X \leq \|y\|_X$ for any $x,y\in X^+$. Otherwise, we choose an equivalent norm with this property, and note that eISS in one norm implies eISS  in any other equivalent norm, and that the spectral radius does not depend on the choice of an equivalent norm.%

Together with the subadditivity of $A^n$, $n\in\N$ this implies for all $x,y \in X^+$ that%
\begin{align}\label{eq_c_subadditive}
\begin{split}
  V(x + y) &= \sup_{n \in \Z_+} \eta^n \| A^n(x + y) \|_X 
	\leq \sup_{n \in \Z_+} \eta^n \|A^n(x) + A^n(y)\|_X \\
	&\leq \sup_{n \in \Z_+} \eta^n (\|A^n(x)\|_X + \|A^n(y)\|_X) \leq V(x) + V(y),
\end{split}
\end{align}
and hence $V$ is subadditive as well. Now consider a sequence $x(\cdot)$ in $X^+$ such that%
\begin{equation}
\label{eq:Discrate-time-sys-aux}
  x(k+1) = A(x(k)) + u(k) \mbox{\quad for all\ } k \in \Z_+.%
\end{equation}
It then follows that%
\begin{align*}
  V(x(k+1)) &= V(A(x(k)) + u(k)) \leq V(A(x(k))) + V(u(k)) \leq \frac{1}{\eta} V(x(k)) + V(u(k)).%
\end{align*}
By \eqref{eq:LF-estimate-above}, we obtain%
\begin{equation*}
  V(x(k+1)) \leq \eta^{-1} V(x(k)) + \psi \|u\|_{\infty} \mbox{\quad for all\ } k \geq 0.%
\end{equation*}

{
As $V$ is homogeneous of degree one, subadditive and monotone,
 $V$ is Lipschitz continuous using the argumentation in the beginning of the proof. 

\q{\ref{itm:eISS-criterion-equiv-3} $\Rightarrow$ \ref{itm:eISS-criterion-equiv-1}}: This standard argument it omitted.
}
	%
\qed%
\end{proof}

{

\section{Systems, governed by a max-form gain operator}
\label{sec:Systems, governed by a max-form gain operator}

Here, we study the properties of the operator $\Gamma_\otimes$ and its strong transitive closure. These results strengthen the corresponding results in \cite[Sec.~4]{DMS19a}, and are motivated by them.
We use these results to characterize the MBI and MLIM properties for the operator $\Gamma_{\otimes}$ with linear gains in 
Proposition~\ref{prop:join-morphism-eISS-criterion}. However, the developments of this section are also useful for the construction of paths of strict decay for the nonlinear operator $\Gamma_{\otimes}$, which is essential for nonlinear Lyapunov-based small-gain theorems.

The powers of the operator $\Gamma_\otimes$ have a particularly simple representation:%

\begin{lemma}
\label{lem:Gamma_otimes_formula} 
For any $n\in\N$ and any $x \in \ell_\infty^+(I)$, it holds that
\begin{equation}\label{eq:Potenzen-join-morphism}
  \Gamma^n_\otimes(x) = \Big(\sup_{j_2,\ldots,j_{n+1}\in I}\gamma_{i j_2}\circ \cdots\circ \gamma_{j_{n}j_{n+1}} (x_{j_{n+1}})\Big)_{i\in I}.
\end{equation}
\end{lemma}

\begin{proof}
For $n=1$ the claim is clear. Let  the claim hold 
 for a certain $n\in\N$. 
Then
\begin{eqnarray*}
  \Gamma^{n+1}_\otimes(x) 
	&=& \Gamma_\otimes(\Gamma^{n}_\otimes(x)) 
	= \big(\sup_{j_2 \in I}\gamma_{ij_2} \circ \sup_{j_3,\ldots,j_{n+2}\in I}\gamma_{j_2 j_3}\circ \cdots\circ \gamma_{j_{n+1}j_{n+2}} (x_{j_{n+2}})\big)_{i \in I}.
\end{eqnarray*}
As $(\gamma_{ij})\subset\Kinf$, \quad
$  \Gamma^{n+1}_\otimes(x) 
	= \big(\sup\limits_{j_2,\ldots,j_{n+2}\in I}\vspace{-4mm}\gamma_{ij_2} \circ \gamma_{j_2 j_3}\circ \cdots\circ \gamma_{j_{n+1}j_{n+2}} (x_{j_{n+2}})\big)_{i \in I}.$
	%
\qed
\end{proof}

\begin{remark}\label{rem_homogen_spectral_radius}
Let $\Gamma_{\otimes}$ be homogeneous, which is the case if and only if all the gains $\gamma_{ij}$ are linear. Then using monotonicity of $\Gamma_{\otimes}$, and the formula \eqref{eq:Potenzen-join-morphism}, we obtain the following expression for the spectral radius of $\Gamma_{\otimes}$:
\begin{eqnarray}
\label{eq:Spectral-radius-join-morphism}
  r(\Gamma_{\otimes}) 
	&=& \lim_{n \rightarrow \infty} \sup_{x \in \ell_\infty^+(I),\; \|x\|_{\ell_\infty(I)} = 1}\|\Gamma_{\otimes}^n(x)\|_{\ell_\infty(I)}^{1/n} = \lim_{n \rightarrow \infty} \|\Gamma_{\otimes}^n({\bf 1})\|_{\ell_\infty(I)}^{1/n}\nonumber\\	
       &=& \lim_{n \rightarrow \infty} \Bigl( \sup_{j_1,\ldots,j_{n+1}\in I} \gamma_{j_1j_2} \cdots \gamma_{j_{n}j_{n+1}}\Bigr)^{1/n}.%
\end{eqnarray}
\end{remark}

The following lemma can be found in \cite[Lem.~4.1]{DMS19a}. We include a rather simple proof for the sake of completeness.%

\begin{lemma}\label{lem_cycle_contr}
Assume that $\Gamma_\otimes$ satisfies the small-gain condition. Then all cycles built from the gains $\gamma_{ij}$ are contractions. That is,%
\begin{equation*}
  \gamma_{i_1i_2} \circ \cdots \circ \gamma_{i_{k-1}i_k}(r) < r%
\end{equation*}
for all $r>0$ if $i_1,\ldots,i_k$ is an arbitrary path with $i_1 = i_k$.
\end{lemma}

\begin{proof}
Assume that the assertion is not satisfied for some cycle:%
\begin{equation*}
  \gamma_{i_1i_2} \circ \cdots \circ \gamma_{i_{k-1}i_k}(r) \geq r,\quad i_1 = i_k,\ r > 0.%
\end{equation*}
Then define the vector%
\begin{align*}
  s &:= re_{i_1} + e_{i_2}\gamma_{i_2i_3} \circ \cdots \circ \gamma_{i_{k-1}i_k}(r) \\
	 &\qquad + e_{i_3}\gamma_{i_3i_4} \circ \cdots \circ \gamma_{i_{k-1}i_k}(r) + \ldots + e_{i_{k-1}}\gamma_{i_{k-1}i_k}(r).%
\end{align*}
We check that $\Gamma_\otimes(s) \geq s$ to obtain a contradiction. Indeed, for any nonzero component of $s$ indexed by $i_{\nu}$ we have (recall that $i_k = i_1$)%
\begin{align*}
  \Gamma_{\otimes,i_{\nu}}(s) &= \sup_{j\in\N}\gamma_{i_{\nu}j}(s_j) \geq \gamma_{i_{\nu}i_{\nu+1}}(s_{i_{\nu+1}}) 
	= \gamma_{i_{\nu}i_{\nu+1}} \circ \cdots \circ \gamma_{i_{k-1}i_k}(r).%
\end{align*}
The last expression equals $s_{i_{\nu}}$ if $\nu \in \{2,\ldots,k-1\}$ and is $\geq r = s_{i_1}$ if $\nu = 1$. This implies $\Gamma_\otimes(s) \geq s$ in contradiction to the small-gain condition.%
\qed
\end{proof}


The implication shown in Lemma~\ref{lem_cycle_contr} cannot be reversed in general. However, if the operator $\Gamma$ is block-diagonal with finite-dimensional components, then the implication can be ``upgraded'' to the following criterion.%

\begin{proposition}
\label{prop:SGC_block-diagonal-operators} 
	%
Assume that the gain matrix $\Gamma =(\gamma_{ij})_{i,j\in\N}$, $\gamma_{ij}\in\Kinf\cup\{0\}$, is a block-diagonal matrix of the form 
$\Gamma = \diag(\Gamma_1, \Gamma_2, \ldots)$, where all $\Gamma_j:\R^{n_j}_+\to \R^{n_j}_+$ and $n_j <\infty$ for all $j\in\N$.
We write for the corresponding gain operators $\Gamma_\otimes = \diag(\Gamma_{1,\otimes}, \Gamma_{2,\otimes}, \ldots)$.
Then the following statements are equivalent:
\begin{itemize}
	\item[(i)]  $\Gamma_\otimes:\ell_\infty^+\to\ell_\infty^+$ satisfies the small-gain condition. 
	\item[(ii)] For all $i\in\N$, the operator $\Gamma_{i,\otimes}$ satisfies the small-gain condition.
	\item[(iii)] All cycles built from the gains of $\Gamma$ are contractions (as defined in Lemma~\ref{lem_cycle_contr}).
	\item[(iv)] All cycles built from the gains of $\Gamma_i$, $i\in\N$, are contractions.
\end{itemize}
\end{proposition}

\begin{proof}
(i) $\Rightarrow$ (ii). Assume that $\Gamma_\otimes$ satisfies the small-gain condition.
Pick any $i\in\N$ and consider the vector $s = (s_1,s_2,\ldots) \in\ell_\infty^+$ with $s_j=0$ for $j\neq i$ and $s_i\neq 0$.

For this operator, we have $(\Gamma_\otimes s)_i = \Gamma_{i,\otimes} s_i$, and $(\Gamma_\otimes s)_j=0$ for $j\neq i$.
As $\Gamma_\otimes s \not\geq s$ for all $s_i\neq 0$, it follows  that $\Gamma_{i,\otimes} s_i \not\geq s_i$, for all $s_i\neq 0$.
As $i\in\N$ has been chosen arbitrarily, the claim follows.

(ii) $\Rightarrow$ (i). Pick any $s\in\ell_\infty^+\backslash\{0\}$. We can write it as $s=(s_1,s_2,\ldots)$.
Pick any $i\in\N$ so that $s_i\neq 0$. As $s\neq 0$, such $i$ exists and it holds that $\Gamma_i s_i \not\geq s_i$.

As $\Gamma_\otimes s = (\Gamma_{1,\otimes}s_1,\Gamma_{2,\otimes} s_2, \ldots)$ and by definition of the order in $\ell_\infty$ 
it holds that $\Gamma_\otimes s \geq s$ $\Iff$ $\Gamma_{i,\otimes} s_i \geq s_i$ for all $i$, it follows that $\Gamma_\otimes s \not\geq s$.

(ii) $\Iff$ (iv). Well-known, see, e.g., \cite[p.~108]{DRW07}.

(iii) $\Iff$ (iv). This follows by noting that all the nonzero cycles of $\Gamma_\otimes$ are cycles of $\Gamma_{i,\otimes}$ for a certain $i\in\N$.
\qed
\end{proof}

We will use the following lemma which is interesting in itself, since it does not require the gains to be linear. This is a strenthening of \cite[Lem.~4.3]{DMS19a}, and is ultimately useful for a construction of so-called paths of strict decay (a.k.a.~$\Omega$-paths) in Lyapunov-based small-gain theorems, see \cite{DMS19a}.%

\begin{lemma}
\label{lem_Q}
Assume that $\Gamma_{\otimes}:\ell_{\infty}^+(I) \rightarrow \ell_{\infty}^+(I)$, defined as in \eqref{eq:Gain-operator-semimax}, is well-defined, continuous and satisfies the robust small-gain condition. Then the operator%
\begin{equation*}
Q:\ell_{\infty}^+(I) \to \ell_{\infty}^+(I),\quad Q(s) := \sup_{k \in \Z_+} \Gamma_{\otimes}^k(s) \mbox{\quad for all\ } s \in \ell_{\infty}^+(I),%
\end{equation*}
where the supremum is taken componentwise, is well-defined and satisfies%
\begin{equation}
\label{eq:upper estimate on Q}
  s \leq Q(s) \leq \omega^{-1}(\|s\|_{\ell_{\infty}}){\bf 1} \mbox{\quad for all\ } s \in \ell_{\infty}^+(I),
\end{equation}
where $\omega\in\Kinf$, $\omega<\id$ stems from the robust small-gain condition.
In particular,
\begin{equation}
\label{eq:Sandwich-estimates-Q}
  \|s\|_{\ell_{\infty}} \leq\|Q(s)\|_{\ell_{\infty}} \leq \omega^{-1}(\|s\|_{\ell_{\infty}}) \mbox{\quad for all\ } s \in \ell_{\infty}^+(I),
\end{equation}
Furthermore,
\begin{equation}
\label{eq_Gamma_Q_ineq}
  \Gamma_{\otimes}(Q(s)) =Q(\Gamma_{\otimes}(s)) \leq Q(s) \mbox{\quad for all\ } s \in \ell_{\infty}^+(I).
\end{equation}
\end{lemma}

\begin{proof}
Considering $k=0$ in the definition of $Q$, it is clear that $Q(s) \geq s$ for all $s\in \ell_\infty^+(I)$.

Now assume that either $Q$ is not well-defined, or $Q(s) > \omega^{-1}(\|s\|_{\ell_{\infty}}){\bf 1}$ for a certain $s \in \ell_\infty^+(I)$.
In any case, this implies that there are $s \in \ell_\infty^+(I)$ and $i \in I$ such that $\sup_{k \in \Z_+} \big[\Gamma_{\otimes}^k(s)\big]_i > \omega^{-1}(\|s\|_{\ell_{\infty}})$.

By formula \eqref{eq:Potenzen-join-morphism}, we obtain that there exist $k \in \N$, indices $i,j \in \N$ and a path $j_1,\ldots,j_k$ such that%
\begin{equation}\label{eq_con_ass}
  \gamma_{ij_1} \circ \gamma_{j_1j_2} \circ \cdots \circ \gamma_{j_kj}(s_j) \geq \omega^{-1}(\|s\|_{\ell_{\infty}}).%
\end{equation}
For the given $i,j$, we define the operator%
\begin{equation*}
  \tilde{\Gamma}_{ji}(s)_l := \Bigl(\sup_{k\in\N}\left[\gamma_{lk}(s_k) + \delta_{jl}\delta_{ik} \omega(s_k)\right]\Bigr)_{l\in\N} \mbox{\quad for all\ } s \in \ell_{\infty}^+(I),%
\end{equation*}
where $\delta_{xy}$ is the Kronecker delta, and observe that%
\begin{align*}
  \tilde{\Gamma}_{ji}(s) &\leq \Bigl(\sup_{k\in\N}\gamma_{lk}(s_k) + \sup_{k\in\N} \delta_{jl}\delta_{ik} \omega(s_k)\Bigr)_{l\in\N} \\
	&= \Bigl(\sup_{k\in\N}\gamma_{lk}(s_k) + \delta_{jl} \omega(s_i)\Bigr)_{l\in\N} = \Gamma_{\otimes}(s) + \omega(s_i) e_j = \Gamma_{ji}(s).%
\end{align*}
Since $\Gamma_{ji}$ satisfies the small-gain condition by assumption, then also $\tilde{\Gamma}_{ji}$ satisfies the small-gain condition $\tilde{\Gamma}_{ji}(s) \not\geq s$ for all $s \in \ell_{\infty}^+(I) \setminus \{0\}$. By Lemma~\ref{lem_cycle_contr},
 all cycles built from the gains%
\begin{equation*}
  \tilde{\gamma}_{lk}(r) := \gamma_{lk}(r) + \delta_{jl}\delta_{ik} \omega(r),\quad (l,k) \in \N^2%
\end{equation*}
are contractions. In particular,%
\begin{equation*}
  \tilde{\gamma}_{ji} \circ \tilde{\gamma}_{ij_1} \circ \tilde{\gamma}_{j_1j_2} \circ \cdots \circ \tilde{\gamma}_{j_kj}(s_j) < s_j.%
\end{equation*}
With \eqref{eq_con_ass}, we thus obtain%
\begin{align*}
  s_j &> \tilde{\gamma}_{ji} \circ \tilde{\gamma}_{ij_1} \circ \tilde{\gamma}_{j_1j_2} \circ \cdots \circ \tilde{\gamma}_{j_kj}(s_j) \geq \tilde{\gamma}_{ji} \circ \gamma_{ij_1} \circ \gamma_{j_1j_2} \circ \cdots \circ \gamma_{j_kj}(s_j) \\
			&= ( \gamma_{ji} + \omega ) \circ \gamma_{ij_1} \circ \gamma_{j_1j_2} \circ \cdots \circ \gamma_{j_kj}(s_j) \geq \|s\|_{\ell_{\infty}} \geq s_j,%
\end{align*}
a contradiction. This shows \eqref{eq:upper estimate on Q}.

The sandwich estimate \eqref{eq:Sandwich-estimates-Q} follows from \eqref{eq:upper estimate on Q} by monotonicity of the norm in $\ell_\infty(I)$.%

Let us show \eqref{eq_Gamma_Q_ineq}. For every $i\in\N$, we have%
\begin{align*}
  [\Gamma_{\otimes}(Q(s))]_i &= \sup_{j\in\N} \gamma_{ij}( Q(s)_j ) = \sup_{j\in\N} \gamma_{ij}( \sup_{k\in\Z_+} \Gamma_{\otimes}^k(s)_j ) \\
	               &= \sup_{j\in\N}\sup_{k\in\Z_+} \gamma_{ij}(\Gamma_{\otimes}^k(s)_j) = \sup_{k\in\Z_+}\sup_{j\in\N} \gamma_{ij}(\Gamma_{\otimes}^k(s)_j) \\
								 &= \sup_{k\in\Z_+}\Gamma_{\otimes}(\Gamma_{\otimes}^k(s))_i = \sup_{k\in\Z_+}\Gamma_{\otimes}^{k+1}(s)_i \leq \sup_{k\in\Z_+} \Gamma_{\otimes}^k(s)_i = Q(s)_i,%
\end{align*}
and hence \eqref{eq_Gamma_Q_ineq} holds. 

Finally, for any $s \in \ell_\infty^+(I)$ we have

$\displaystyle Q(\Gamma_{\otimes}(s))  
= \sup_{k \in \Z_+} \Gamma_{\otimes}^{k+1}(s)
= \sup_{k \in \Z_+} \Gamma_{\otimes} (\Gamma_{\otimes}^{k}(s))
= \Gamma_{\otimes} (\sup_{k \in \Z_+} (\Gamma_{\otimes}^{k}(s)))
= \Gamma_{\otimes} (Q(s)).$
\qed%
\end{proof}

\begin{remark}
\label{rem:Kleene-star} 
The operator $Q$ is of fundamental importance in max-algebra, and is sometimes called the strong transitive closure of $\Gamma_\otimes$, or Kleene star, see \cite[Sec.~1.6.2]{But10}, or just the closure of $\Gamma_\otimes$, see \cite[Sec.~1.4]{Rue17}.
%
\end{remark}

}

\end{appendices}

\end{document}